\documentclass[11pt]{article}
\usepackage{makeidx}
\usepackage{amsmath,amssymb}
\usepackage{amsthm}
\usepackage{graphicx}
\usepackage{color}
\usepackage{verbatim}
\usepackage{algorithm}
\usepackage{longtable}

\usepackage{microtype}

\definecolor{blue}{rgb}{0,0,0.9}
\definecolor{red}{rgb}{0.9,0,0}
\definecolor{green}{rgb}{0,0.9,0}

\newcommand{\cX}{{\cal X}}
\newcommand{\cO}{{\cal O}}
\newcommand{\cY}{{\cal Y}}
\newcommand{\cT}{{\cal T}}
\newcommand{\cJ}{{\cal J}}

\newcommand{\cH}{{\cal H}}
\newcommand{\cM}{{\cal M}}
\newcommand{\cN}{{\cal N}}
\newcommand{\cP}{{\cal P}}

\newcommand{\cA}{{\cal A}}

\newcommand{\cQ}{{\cal Q}}

\newcommand{\cI}{{\cal I}}

\newcommand{\inprod}[2]{\langle #1 , #2 \rangle }

\newcommand{\bc}{\begin{center}}
\newcommand{\ec}{\end{center}}

\newcommand{\R}{\mathbb R}

\newcommand{\be}{\begin{equation}}
\newcommand{\ee}{\end{equation}}
\newcommand{\beaa}{\begin{eqnarray*}}
\newcommand{\eeaa}{\end{eqnarray*}}
\newcommand{\ben}{\begin{enumerate}}
\newcommand{\een}{\end{enumerate}}
\newcommand{\db}{\hspace*{\fill}{\zapf o}}

\newcommand{\bpn}{\begin{proposition}\twlsf}
\newcommand{\epn}{\db\end{proposition}}
\newcommand{\bdm}{\begin{displaymath}}
\newcommand{\edm}{\end{displaymath}}
\newcommand{\ba}{\begin{array}}
\newcommand{\ea}{\end{array}}

\newcommand{\argmin}{\mathop{\rm argmin}}

\def\texitem#1{\par\smallskip\noindent\hangindent 25pt
               \hbox to 25pt {\hss #1 ~}\ignorespaces}
\newcommand{\norm}[1]{\left\lVert#1\right\rVert}

\newtheorem{assumption}{Assumption}

\newtheorem{lemma}{Lemma}
\newtheorem{proposition}{Proposition}

\newtheorem{remark}{Remark}
\newtheorem{theorem}{Theorem}

\newcommand{\eps}{\epsilon}

\def\mc{\multicolumn}
 
\def\sig{\sigma}

\def\inprod#1#2{\langle#1,\,#2\rangle}
\def\norm#1{\|#1\|}

\def\grad{\nabla}
\def\gam{\gamma}

 \def\cL{{\cal L}} \def\cI{{\cal I}} 
 \def\cQ{{\cal Q}}

\def\P{\mathbb{P}}

\def\ome{\omega}

\DeclareMathOperator*{\minimax}{minimax}

\begin{document}

\title{ An asymptotically {superlinearly convergent semismooth Newton
augmented Lagrangian method for Linear Programming}}
\date{March 19, 2020}

\author{Xudong Li\thanks{School of Data Science, Fudan University, Shanghai, China ({\tt lixudong@fudan.edu.cn}); Shanghai Center for Mathematical Sciences, Fudan University, Shanghai, China.	
}%
	\and 
	Defeng Sun\thanks{Department of Applied Mathematics, The Hong Kong Polytechnic University, Hung Hom,
		Hong Kong ({\tt defeng.sun@polyu.edu.hk}).
	}
\and \ Kim-Chuan Toh\thanks{Department of Mathematics, and Institute of Operations Research and Analytics, National University of Singapore, Singapore
({\tt mattohkc@nus.edu.sg}).
This author's research is partially supported by the Academic Research Fund of the Ministry of Singapore
under Grant R146-000-257-112. 
}}

\maketitle

\begin{abstract}
Powerful interior-point methods (IPM) based commercial solvers, such as Gurobi and Mosek, have been hugely successful in solving large-scale linear programming {(LP) problems}. The high efficiency of these solvers depends critically on the sparsity of the problem data and advanced matrix factorization techniques.
For a large scale LP problem with data matrix $A$ that is dense (possibly structured) or whose corresponding
normal matrix $AA^T$ has a dense Cholesky factor (even with re-ordering), these solvers may require excessive computational cost and/or extremely heavy memory usage in each interior-point iteration.
Unfortunately, the natural remedy, i.e., the use of iterative methods based IPM solvers, although can avoid the explicit computation of the coefficient matrix and its factorization,  is not practically viable due to the inherent
extreme ill-conditioning of the large scale normal equation arising in each interior-point iteration. To provide a better
alternative choice for  solving large scale LPs with dense data or requiring expensive factorization of its normal equation, we propose a semismooth Newton based inexact proximal augmented Lagrangian ({\sc Snipal}) method.
Different from  classical IPMs, in each iteration of {\sc Snipal}, iterative methods can
efficiently be used to solve simpler yet better conditioned semismooth Newton linear systems. Moreover, {\sc Snipal} not only enjoys a fast  asymptotic  superlinear convergence but {is also proven to enjoy a} finite termination property. Numerical comparisons with Gurobi have demonstrated encouraging potential of {\sc Snipal} for handling large-scale LP problems where the constraint matrix $A$ has a dense  representation or $AA^T$ has a
dense factorization even with an appropriate re-ordering.
For a few large LP instances arising from correlation clustering,
our algorithm can be up to {$20-100$} times faster than 
{the barrier method implemented in Gurobi for solving
the problems to the accuracy of $10^{-8}$ in the
relative KKT residual. {However, when tested on some large sparse LP problems available in
the public domain, our algorithm is not yet practically competitive against the barrier method in Gurobi, especially
when the latter can compute the Schur complement matrix and its sparse Cholesky factorization in each iteration cheaply.}
}
\end{abstract}



\section{Introduction}

It is well known that primal-dual interior-point methods (IPMs) as implemented
in highly optimized commercial solvers, such as Gurobi and Mosek, are powerful methods
for solving large scale linear programming (LP) problems with conducive sparsity. However, 
the large scale normal (also called Schur complement) equation arising in each interior-point iteration is generally highly ill-conditioned 
when the barrier parameter is small, and typically it is necessary to employ
 a direct method, such as the sparse Cholesky factorization, to solve
the equation stably and accurately. Various attempts, {for example}
in \cite{BGZ-04, Chai-Toh-07,CMTH-16,Gondzio-08,Oliveira-05},
 have been made in using an iterative solver, such as the preconditioned conjugate-gradient (PCG) method,
to solve the normal equation when it is too expensive to compute
{the coefficient matrix or}
the sparse Cholesky
factorization because of excessive computing time or memory usage due to fill-ins.
For more details on the numerical performance of iterative methods based IPMs
for solving large scale LP, we refer the readers to \cite{CMTH-16} and the references therein.
However, the
extreme ill-conditioning of the normal equation (and also of the augmented equation)
makes it extremely costly for an iterative method to solve the equation either because
it takes excessive number of steps to converge or because constructing an effective
preconditioner is prohibitively expensive.
{For a long time since their inceptions}, iterative methods based IPMs have
not been  proven convincingly to be more efficient in general than the
highly powerful solvers, such as Gurobi and Mosek,
on various large scale LP test instances.
{Fortunately, recent promising progress has been made in
the work of Schork and Gondzio \cite{Schork-Gondzio} where the authors
proposed effective basis matrix preconditioners for iterative methods based IPMs. 
However, we should note that as the construction of the basis matrix preconditioners in \cite{Schork-Gondzio}  
requires the explicit storage of the constraint matrix $A$, the approach may not be
applicable to the case when $A$ is not explicitly given but
defined via a linear map. In contrast,  the algorithm designed
in this paper is still applicable under the latter scenario.}
{For later discussion, here we give an example where $A$ is defined by a linear map: $A\in \R^{n^2} \to \R^{p^2}$
such that $Ax = {\rm vec} (B\, {\rm mat}(x)\, D^T)$, where  $B,D\in \R^{p\times n}$ are given matrices, ${\rm mat}(x)$  denotes the
operation of converting a vector $x\in\R^{n^2}$ into an $n\times n$ matrix and
${\rm vec}(X)$ denotes the operation of converting a matrix $X\in \R^{p\times p}$ into a $p^2$-dimensional vector.
It is easy to see that the matrix representation of $A$ is the kronecker product $D\otimes B$, and it could be 
extremely costly to store {$D\otimes B$} explicitly when $B,D$ are large dimensional dense matrices. 
}

The goal of this paper is to design a semismooth Newton inexact proximal augmented Lagrangian ({\sc Snipal}) method
for solving large scale LP problems, which {has} the following key properties:  (a) The {\sc Snipal} method
can achieve fast local linear convergence; (b) The semismooth Newton equation
arising in each iteration can fully exploit {the solution sparsity in addition to data
sparsity}; (c) The semismooth Newton equation
is typically much better conditioned than its counterparts
in IPMs, even when the iterates approach optimality.
The latter two properties {thus} make it cost effective for one to use an iterative method, such as the PCG method,
to solve
the aforementioned linear system when it is large.
It is these three key properties that give the competitive advantage of our {\sc Snipal} method
over the highly developed IPMs for solving certain classes of large scale LP problems
which we will describe shortly.

 Consider the following primal and dual LP problems:
\begin{eqnarray*}
\mbox{(P)} & \min\Big\{ c^T x + \delta_K(x) \mid Ax = b ,\; x\in \R^n\Big\}&
\\[5pt]
\mbox{(D)} &\max\Big\{ -\delta_K^* (A^*y - c) + b^Ty \mid y\in \R^m \Big\} &
\end{eqnarray*}
where $A\in \R^{m\times n}$, $b\in \R^m$, $c\in \R^n$ are given data.
The set $K = \{ x\in\R^n\mid l\leq x \leq u \}$ is
 a simple polyhedral set,
where $l,u$ are given vectors.
We allow the components of $l$ and $u$ to be $-\infty$ and $\infty$, respectively.
In particular, $K$ can model the nonnegative {orthant} $\R^n_+$.
In the above, $\delta_K(\cdot)$ denotes the
indicator function over the set $K$ such that $\delta_K(x) = 0$ if $x\in K$ and
$\delta_K(x) = \infty$ otherwise. The Fenchel conjugate of $\delta_K$ is denoted
by $\delta_K^*$.
{We note that while we focus on the indicator function $\delta_K(\cdot)$ in (P), the
algorithm and theoretical results we have developed in this paper are also
applicable when $\delta_K$ is replaced by a closed convex polyhedral function $p :\R^n \to (-\infty,\infty]$.}
We  {make} the following assumption on the problems (P) and (D).
\begin{assumption}
	\label{assump:fesPD}
The solution set of (P) and (D) is nonempty and
 $A$ has full row rank (hence $m\leq n$).
\end{assumption}

Our {\sc Snipal} method is designed for the dual LP but the primal variable is also
generated in each iteration.
In order for the fast local convergence property to kick-in early, we warm-start the {\sc Snipal} method
by an alternating direction method of multipliers (ADMM), which is also applied to the dual LP.
We should mention that our goal is not to use {\sc Snipal} as a general purpose solver for LP
but to complement  the excellent general solvers (Gurobi and Mosek)
when {the latter} are too expensive  or have difficulties in solving very large scale problems
due to memory limitation.
In particular, we are interested in solving large scale LP problems  having one of the following characteristics.
\begin{enumerate}
\item The number of variables $n$ in (P)
is significantly larger than the number of linear constraints $m$. We {note that
such a property is not restrictive since for a primal problem with a huge number of inequality
constraints $Ax \leq b$ and $m \gg n$}, we can treat the
dual problem (D) as the primal LP, and the required property is satisfied.

\item The constraint matrix $A$ is large and dense but it has an economical representation such as
being the Kronecker product of two matrices, {or $A$ is sparse but $AA^T$ has
a dense {factorization  even with an appropriate re-ordering}}. For such an LP problem, it may not be possible to solve it by
using the standard interior-point methods  implemented in Gurobi or Mosek since $A$ cannot be
stored explicitly. {Instead}, one would need to use a Krylov subspace iterative method
to solve the underlying large and dense linear system of equations arising in each iteration of an IPM or
{\sc Snipal}.

\end{enumerate}

In \cite{Wright-90}, Wright proposed an algorithm for solving the primal problem (P) for the
special case where $K=\R^n_+$.
The {proposed} method is in fact the
{proximal method of multipliers}
applied to (P) while keeping the
nonnegative constraint in the quadratic programming (QP) subproblem. More specifically, suppose {that the iterate} at the $k$th iteration is
$(x^k,y^k)$ and the penalty parameter is $\gam_k = \sig_k^{-1}$. Then the QP subproblem
is given by $\min \{ \frac{1}{2}\inprod{(\sig_k A^*A + \sig_k^{-1} I_n)x}{x} +
\inprod{x}{c - A^* y^k -\sig_k^{-1} x^k - \sig_k A^* b} \mid x \geq 0\}$.
In \cite{Wright-90}, an SOR (successive over-relaxation) method is used to solve the
QP subproblem. But it is unclear how this subproblem can be solved efficiently
when $n$ is large.
In contrast, in this paper, we propose a semismooth Newton based inexact proximal augmented Lagrangian  {({\sc Snipal})} method that  is  applied to the dual problem (D) and the
associated subproblems are solved efficiently by a semismooth Newton method {having at least {local} superlinear convegence or even} quadratic
convergence.

In the pioneering work of De Leone and Mangasarian \cite{Mangasarian-88},
an augmented Lagrangian method is applied to an equivalent reformulation of (D), and the QP subproblem
of the form $\min\{ -b^T y + \frac{\sig}{2}\norm{A^* y + z -c + \sig^{-1} x^k}^2 \mid y\in\R^m, z\geq 0 \}$
in each iteration is solved by a projected SOR method.
Interestingly, in a later paper \cite{Mangasarian-04}, {based on the results obtained in \cite{Mangasarian-79}}, Mangasarian designed
a generalized Newton method to first solve
a penalty problem of the form
$\min \{ -\eps b^T y + \frac{1}{2}\norm{\Pi_{\R^n_+}(A^* y - c)}^2\}$
and then use its solution
to indirectly solve (P) for $K=\R^n_+$, under the condition that the positive parameter
$\eps$ must be below a certain unknown threshold and a strong uniqueness condition holds.  Soon after,
\cite{EGM-05} observed that the restriction on the parameter in \cite{Mangasarian-04}
can be avoided by modifying the procedure in \cite{Mangasarian-04}
via the augmented Lagrangian method {but} the corresponding subproblem in each
iteration must be solved {\em exactly}.
As the generalized Newton system is likely to be singular,
in both \cite{Mangasarian-04} and \cite{EGM-05}, the system is modified by adding a scalar multiple of the identity matrix to the generalized Hessian. {Such a perturbation, however, would} destroy the fast local convergence property of the generalized
Newton method. {We also note that to obtain the minimum norm solution of the primal problem (P), \cite{Kanzow2003minimum} proposed a generalized Newton method for solving $\min \{\frac{1}{2}\norm{\Pi_{\R_+^n}(A^*y - rc)}^2 - \inprod{b}{y} \}$ with the positive parameter $r$ being sufficiently large. Although \cite{Kanzow2003minimum} contains no computational results, the authors obtained the global convergence and finite termination properties of the proposed method under the assumption that the Newton linear systems involved 
 are solved exactly and a certain regularity condition on the nonsingularity of generalized Jacobians holds.} {More recently, \cite{YZHRD} designed an ALM for the primal problem (P)
for which a  bound-constrained convex QP subproblem must be solved in each iteration. In the paper,
this subproblem is solved by a randomized coordinate descent (RCD) method with an active set implementation.
{There are several drawbacks} to this approach. First, solving the QP subproblem can be {time consuming}.
Second, the RCD approach can hardly exploit any specific structure of the matrix $A$ 
{(for example, when $A$ is defined by the Kronecker product of two given matrices)} to speed
up the computation of the QP subproblem. Finally, it also
does not exploit  the sparsity structure {present}} in the Hessian of the underlying QP subproblem
to speed up the computation.

Here, we employ the an inexact proximal augmented Lagrangian (PAL) method
to (D) to simultaneously solve (P) and (D).
Our entire algorithmic design is dictated by the focus on computational efficiency and generality.
From this perspective,
now we elaborate on the key differences between our paper and \cite{EGM-05}.
{First, without any reformulation, our algorithm is directly applicable to problems with a more general set $K$ instead of just $\R^n_+$ as in \cite{Mangasarian-04} and \cite{EGM-05}.}
Second, we use the  inexact PAL
framework which ensures that in each iteration,
an unconstrained minimization subproblem involving the variable $y$ is strongly convex
and hence the semismooth Newton method we employ to solve this subproblem
can attain local quadratic convergence.
Third, the flexibility of allowing the PAL
subproblems to be solved inexactly
can lead to substantial computational savings, especially during the initial phase
of the algorithm. {Fourth,  for computational efficiency,
we warm-start our inexact PAL method by using a first-order method.}
Finally, as solving the semismooth Newton linear systems
is the most critical component of the entire algorithm,
we have devoted a substantial part of the paper on proposing novel
numerical strategies to solve the linear systems efficiently.

{Numerical comparisons of our semismooth Newton proximal augmented Lagrangian method
({\sc Snipal})
with the barrier method in 
Gurobi have demonstrated encouraging potential of our method for handling large-scale LP problems where the constraint matrix $A$ has a dense  representation or $AA^T$ has a
dense factorization even with an appropriate re-ordering.
For a few large LP instances arising from correlation clustering,
our algorithm can be up to {$20-100$} times faster than 
the barrier method implemented in Gurobi for solving
the problems to the accuracy of $10^{-8}$ in the
relative KKT residual. However, when tested on some large sparse LP problems available in
the MIPLIB2010 \cite{MIPLIB2010}, our algorithm is not yet practically competitive against the barrier method in Gurobi, especially
when the latter can compute the Schur complement matrix and its sparse Cholesky factorization in each iteration cheaply.
}

\medskip
{The remaining part of the paper is organized as follows. 
In the next section, we introduce a preconditioned proximal point algorithm (PPA) and establish its global and 
local (asymptotic) superlinear convergence. In section 3, we develop a semismooth Newton proximal augmented Lagrangian
method for solving the dual LP (D), and derive its connection to the preconditioned PPA. 
Section 4 is devoted to developing numerical techniques for solving the 
linear system of equations in the semismooth Newton method employed to solve the 
subproblem  in each proximal augmented Lagrangian iteration. 
We describe how to employ an ADMM to warm-start the proximal augmented Lagrangian method in section 5. 
In section 6, we evaluate the numerical performance of our algorithm (called {\sc Snipal}) against the
barrier method in Gurobi on various classes of large scale LPs, including some large sparse
LPs available in the public domain. 
We conclude the paper in the final section. 
}

\medskip
\noindent{\bf Notation.} We use $\cX$ and $\cY$ to denote finite dimensional real Euclidean spaces each endowed with an inner product
$\inprod{\cdot}{\cdot}$ and its induced norm $\norm{\cdot}$. For any self-adjoint positive semidefinite linear operator $\cM:\cX\to \cX$, we define $\inprod{x}{x'}_{\cM} := \inprod{x}{\cM x'}$ and $\norm{x}_{\cM}: = \sqrt{\inprod{x}{\cM x}}$ for all $x,x'\in\cX$.
The largest eigenvalue of $\cM$ is denoted by $\lambda_{\max}(\cM)$. A similar notation is used when $\cM$ is replaced by a matrix $M$.
Let $D$ be a given subset of $\cX$. We write the weighted distance of $x\in\cX$ to $D$ by ${\rm dist}_{\cM}(x,D):= \inf_{x'\in D}\norm{x-x'}_{\cM}$. If $\cM$ is the identity operator, we just omit it from the notation so that ${\rm dist}(\cdot, D)$ is the Euclidean distance function. If $D$ is closed, the Euclidean projector over $D$ is defined by $\Pi_{D}(x): = \argmin\{\norm{x - d}\mid d\in D \}$.
Let $F:\cX \rightrightarrows \cY$ be a multivalued mapping. We define the graph of $F$ to be the set ${\rm gph}F:= \{
(x,y)\in\cX\times\cY\mid y\in F(x)
\}.$ The range of a multifunction is defined by ${\rm Range}(F):=\{
y\mid \exists \,x \mbox{ with } y\in F(x)
\}$.

\section{A preconditioned proximal point algorithm}
\label{sec:mPPA}

In this section, we present a {preconditioned} proximal point algorithm (PPA) and study its convergence properties. In particular, following the classical framework developed in \cite{rockafellar1976monotone,rockafellar1976augmented}, we  prove the global convergence of {the} preconditioned PPA. Under {a} mild error bound condition, global linear rate convergence is also derived. In fact, by choosing the parameter $c_k$ {in the algorithm to be} sufficiently large, the linear rate can be as fast as we please. We further show in Section \ref{sec:globalCon-snipal} that our main Algorithm {\sc Snipal} is in fact an application of the preconditioned proximal point algorithm. Hence, {\sc Snipal}'s convergence properties can be obtained as a direct application of {the general theory} developed here.

Let $\cX$ and $\cY$ be finite dimensional Hilbert spaces and $\cT:\cX\to \cX$ be a maximal monotone operator. Throughout this section, we assume that
$\Omega : = \cT^{-1}(0)$ is nonempty. We further note from \cite[Excerise 12.8]{rockafellar2009variational} that $\Omega$ is a closed set.
The preconditioned proximal point algorithm generates for any start point $z^0\in \cX$ a sequence $\{ z^k\} \subseteq \cX$ by the following approximate rule:
\begin{equation}\label{alg:mPPA}
z^{k+1} \approx \cP_k(z^k), \quad \mbox{where} \quad \cP_k = (\cM_k + c_k \cT)^{-1} \cM_k.
\end{equation}
Here $\{c_k\}$ and $\{ \cM_k \}$ are some sequences of positive real numbers and self-adjoint positive definite linear operators over $\cX$. If $\cM_k \equiv \cI$ for all $k\ge 0$, the updating scheme \eqref{alg:mPPA} recovers the classical proximal point algorithm considered in \cite{rockafellar1976monotone}. Since $\cM_k + c_k \cT$ is a strongly monotone operator, we know from \cite[Proposition 12.54]{rockafellar2009variational} that $\cP_k$ is single-valued and is globally Lipschitz continuous. Here, we further assume that $\{c_k\}$ bounded away from zero, and
{
\begin{equation}
\label{eq:condMk}
(1 + \nu_k) \cM_k \succeq \cM_{k+1}, \quad \cM_k \succeq \lambda_{\min} \cI \quad \forall\, k\ge 0 \quad \mbox{ and } \limsup_{k\to \infty} \lambda_{\max}(\cM_k) = \lambda_{\infty}
\end{equation}
with some nonnegative summarable sequence $\{\nu_k\}$ and constants $+\infty > \lambda_{\infty} \ge \lambda_{\min} >0$. The same condition on $\cM_k$ is also used in \cite{Parente2008class} and can be easily satisfied. For example, it holds obviously if we set $\lambda_{\infty}\cI \succeq \cM_k \succeq \lambda_{\min} \cI$ and $\cM_k \succeq \cM_{k+1}$ for all $k\ge 0$.} 
Note that if $\cT$ is a linear operator, one may rewrite $\cP_k$ as
$\cP_k = (\cI + c_k \cM_k^{-1}\cT)^{-1}$.
We show in the next lemma that this expression in fact holds even for a general {maximal} monotone operator $\cT$. Therefore, we can regard the self-adjoint positive definite linear operator $\cM_k$ as a preconditioner for the maximal monotone operator $\cT$. Based on this observation, we name the algorithm described in \eqref{alg:mPPA} as the preconditioned proximal point algorithm.

\begin{lemma}
	\label{lemma:precondPPA}
Given a constant $\alpha >0$, a self-adjoint positive definite linear operator $\cM$ and a maximal monotone operator $\cT$ on $\cX$, it holds that ${\rm Range}(\cI + \alpha \cM^{-1}\cT) = \cX$ and $(\cI + \alpha \cM^{-1}\cT)^{-1}$ is a single-valued mapping. In addition,
\[
(\cM + \alpha\cT)^{-1}\cM = (\cI + \alpha\cM^{-1}\cT)^{-1}.
\]
\end{lemma}
\begin{proof}
{By \cite[Proposition 20.24]{bauschke2011convex}, we know that $\cM^{-1}\cT$ is maximally monotone. Hence,  ${\rm Range}(\cI + \alpha \cM^{-1}\cT) = \cX$  and $(\cI + \alpha \cM^{-1}\cT)^{-1}$ is a single-valued mapping from $\cX$ to itself.

Now, for any given $z\in \cX$, suppose that $z_1 =(\cI + \alpha \cM^{-1}\cT)^{-1}(z)$. Then, it holds that
\[
\cM z \in (\cM + \alpha \cT) z_1.
\]
Since $(\cM + \alpha \cT)^{-1}$ is a single-valued operator \cite[Proposition 12.54]{rockafellar2009variational}, we know that
\[
z_1 = (\cM + \alpha \cT)^{-1}\cM z,
\]
i.e., $(\cI + \alpha \cM^{-1}\cT)^{-1} z = (\cM + \alpha \cT)^{-1}\cM z$ for all $z\in\cX$.
Thus we have proved the desired equation.}
\end{proof}

{In the literature, the updating scheme \eqref{alg:mPPA} is closely related to {the so-called}
``variable metric proximal point algorithms'';  see for examples \cite{bonnans1995family,Burke1999local,Burke1999variable,Burke2000superlinear,Chen1999proximal,Parente2008class,Qi1995preconditioning}. Among these papers, \cite{bonnans1995family,Chen1999proximal,Qi1995preconditioning} focus only on the case of optimization, i.e., the maximal monotone operator $\cT$ is the subdifferential mapping of a convex function. In addition, they emphasize more on the combination of the proximal point algorithm with quasi Newton method. In \cite{Burke1999local} and the subsequent papers \cite{Burke1999variable,Burke2000superlinear}, {the authors deal with a} general maximal monotone operator $\cT$ and study the following scheme in the exact setting:
	\begin{equation}
	\label{eq:burke}
	z^{k+1} = z^k + \cM_{k}\big((\cI + c_k\cT)^{-1} - \cI \big) z^k.
	\end{equation}
	The global convergence of the scheme \eqref{eq:burke} requires a
	rather restrictive assumption on $\cM_k$ \cite[Hypothesis (H2)]{Burke1999local}, although $\cM_k$ is not required to be {self-adjoint}. In fact, the authors essentially assumed that the deviation of $\cM_k$ from the identity operator should {be small}, and the verification of the assumption can be quite difficult.
	As far as we aware of, \cite{Parente2008class} may be the most related work {to ours}. In \cite{Parente2008class}, the authors consider a variable metric hybrid inexact proximal point method whose updating rule consists {of} an inexact proximal step and a projection step. Moreover, some specially designed {stopping criteria  for the {inexact} solution of the proximal subproblem are also used}. However, due to the extra projection step, the connection between their algorithm and the proximal method of multipliers
	\cite{rockafellar1976augmented} is no longer available.
	Therefore, the results derived in \cite{Parente2008class} cannot be directly used to analyze the convergence properties of  {\sc Snipal} {proposed in this paper,} which is a variant of the proximal method of multipliers. 
	{We should also mention that in \cite{eckstein1993nonlinear}, Eckstein discussed nonlinear proximal point algorithms using Bregman functions and the preconditioned PPA (1) may be viewed as a special instance if  $\cM_k$ is fixed for all $k$. However, the 
algorithms and convergence results in \cite{eckstein1993nonlinear} are not applicable 
to our setting where the linear operator $\cM_k$ can change across iterations.}
	Since the scheme \eqref{alg:mPPA} under the {classical} setting of \cite{rockafellar1976monotone,rockafellar1976augmented} fits our context best, we conduct a comprehensive analysis of its convergence properties which, to   our best knowledge, are currently not available in the literature.
}

For all $k\ge 0$, define the mapping $\cQ_k:= \cI - \cP_k$. Clearly,
if $0\in\cT(z)$, we have that $\cP_k(z) = z$ and $\cQ_k(z) = 0$ for all $k\ge 0$.
Similar to \cite[Proposition 1]{rockafellar1976monotone}, we summarize the properties of $\cP_k$ and $\cQ_k$ in the following proposition:
\begin{proposition}
	\label{prop:pkqk}
	It holds for all $k\ge 0$ that:
	\begin{enumerate}
		\item[(a)] $z = \cP_k(z) + \cQ_k(z)$ and $c_k^{-1}\cM_k\cQ_k(z)\in \cT(\cP_k(z))$ for all $z\in\cX$;
		\item[(b)] $\inprod{\cP_k(z) - \cP_{k}(z')}{\cQ_k(z) - \cQ_k(z')}_{\cM_k} \ge 0$ for all $z,z'\in\cX$;
		\item[(c)] $\norm{\cP_k(z) - \cP_k(z')}_{\cM_k}^2 + \norm{\cQ_k(z) - {\cQ_k(z')}}_{\cM_k}^2 \le \norm{z - z'}_{\cM_k}^2$ for all $z,z'\in\cX$.
	\end{enumerate}
\end{proposition}
\begin{proof}
	The proof can be obtained via simple calculations and is similar to the proof of \cite[Proposition 1]{rockafellar1976monotone}. We omit the details here.
\end{proof}
We list the following two general criteria for the approximate calculation of $\cP_k(z^k)$ which are analogous to those proposed in \cite{rockafellar1976monotone}:
\begin{align*}
& \mbox{(A)}\quad  \norm{z^{k+1} - \cP_k(z^k)}_{\cM_k} \le \epsilon_k,\quad 0 \le \epsilon_k,  \quad
\mbox{$\sum_{k=0}^{\infty}$} \epsilon_k < \infty,
\\[5pt]
&\mbox{(B)}\quad  \norm{z^{k+1} - \cP_k(z^k)}_{\cM_k} \le \delta_k\norm{z^{k+1} - z^k}_{\cM_k}, \quad 0\le \delta_k < 1, \quad \mbox{$\sum_{k=0}^{\infty}$} \delta_k < \infty.
\end{align*}

{
\begin{theorem}
	\label{thm:mPPA}
	Suppose that $\Omega= \cT^{-1}(0)\neq \emptyset$. Let $\{z^k\}$ be any sequence generated by the mPPA \eqref{alg:mPPA} under criterion (A). Then $\{z^k\}$ is bounded and
	\begin{equation}\label{eq:distzk}
{\rm dist}_{\cM_{k+1}}(z^{k+1},\Omega) \le (1 + \nu_k){\rm dist}_{\cM_{k}}(z^{k},\Omega) + (1 + \nu_k)\epsilon_k \quad \forall k\ge0.
	\end{equation}
In addition, $\{z^k\}$ converges to a point $z^{\infty}$ {such that} $0\in \cT(z^{\infty})$.
\end{theorem}
}
\begin{proof}%
	Let $\bar z\in\cX$ be a point satisfying $0\in\cT(\bar z)$. {It is readily shown that $\bar{z} = \cP_k(\bar{z}).$}
	We have
	\begin{equation}
	\label{eq:zkp1Mk}
	\norm{z^{k+1} - \bar z}_{\cM_k} - \epsilon_k \le \norm{\cP_k (z^k) - \bar z}_{\cM_k} = \norm{\cP_k (z^k) - \cP_k(\bar z)}_{\cM_k}
	\le \norm{z^k - \bar z}_{\cM_k}.
	\end{equation}
	Since {$(1 + \nu_k)\cM_{k} \succeq \cM_{k+1}$}, we know that
	\begin{equation}\label{eq:telzk}
	\norm{z^{k+1} - \bar z}_{\cM_{k+1}} 
	\le {(1+\nu_k)\norm{z^{k+1}-\bar{z}}_{\cM_k}}
	\le (1 + \nu_k)\norm{{z^{k}} - \bar z}_{\cM_k} + (1 + \nu_k) \epsilon_k.
	\end{equation}
	Let $\Pi_\Omega(z)$ denotes the projection of $z$ onto $\Omega$. By noting that
	$0 \in \cT(\Pi_\Omega(z^k))$, we get from the above inequality (by setting
	$\bar{z} = \Pi_\Omega(z^k)$) that
	\begin{align*}
	{\rm dist}_{\cM_{k+1}}(z^{k+1},\Omega)
	\le{}& \norm{z^{k+1} - \Pi_{\Omega}(z^k)}_{\cM_{k+1}}\\[5pt]
	\le{}& (1+\nu_k)\norm{z^k - \Pi_{\Omega}(z^k)}_{\cM_k}+ (1 + \nu_k)\epsilon_k \\[5pt]
	={}& (1+\nu_k){\rm dist}_{\cM_k}(z^k,\Omega) +(1 + \nu_k)\epsilon_k.
	\end{align*}
	Since \[
	\sum_{k=0}^{\infty} (1 + \nu_k)\epsilon_k \le \sum_{k=0}^{\infty} \epsilon_k + (\max_{k\ge 0}\epsilon_k) \sum_{k=0}^{\infty} \nu_k < +\infty,
	\]we know from \cite[Lemma 2.2.2]{polyak1987introduction}, \eqref{eq:zkp1Mk} and \eqref{eq:telzk} that
	\begin{equation}
	\label{eq:barz}
	\lim_{k\to \infty}\norm{z^k - \bar z}_{\cM_k} = \lim_{k\to \infty}\norm{z^{k+1} - \bar z}_{\cM_k} = \mu < \infty \quad  \mbox{ and } \quad \lim_{k\to \infty}\norm{{\cP_k(z^k)} - \bar z}_{\cM_k} = \mu.
	\end{equation}
	The boundedness of $\{z^k\}$ thus follows directly from the fact that $\cM_k \succeq \lambda_{\min} \cI$ for all $k\ge 0$. Therefore, $\{z^k\}$ has at least one cluster point $z^\infty$.
	
	From Proposition \ref{prop:pkqk}, we know that for all $k\ge 0$
	\begin{equation}
	\label{eq:Qszpz}
	0\le  \norm{\cQ_k(z^k)}_{\cM_k}^2 \le \norm{z^k - \bar z}_{\cM_k}^2 - \norm{\cP_k(z^k) - \bar z}_{\cM_k}^2.
	\end{equation}
	Therefore, $\lim_{k\to \infty}\norm{\cQ_k(z^k)}_{\cM_k}^2 = 0$.
	It follows that
	\begin{equation}
	\label{eq:cqk}
	\lim_{k\to \infty} c_k^{-1}\cM_k \cQ_k(z^k) = \lim_{k\to \infty} \cQ_k(z^k) = 0,
	\end{equation}
	because the number $c_k$ is bounded away from zero and $\cM_k \succeq \lambda_{\min} \cI$ for all $k\ge0$.
	Since
	\[
	\norm{\cQ_k(z^k)}_{\cM_k} = \norm{(z^k - z^{k+1}) + (z^{k+1} - \cP_k(z^k))}_{\cM_k}
	\ge \norm{z^k - z^{k+1}}_{\cM_k} - \epsilon_k,
	\]
	we further have $\lim_{k\to \infty} {\norm{z^k - z^{k+1}}} = 0$.
	
	Since $z^\infty$ is a cluster point of ${z^k}$ and \[\lim_{k\to \infty}
	\norm{\cP_{k}(z^k) - z^{k+1}} = \lim_{k\to \infty} \norm{z^{k+1} - z^k} = 0,\]
	$z^\infty$ is also a cluster point of $\cP_k(z^k)$.
	From Proposition \ref{prop:pkqk} (a), we have that for any $w\in\cT(z)$
	\[  0\le \inprod{z - \cP_k(z^k)}{w - c_k^{-1}\cM_k\cQ_k(z^k)}\quad \forall\, k\ge 0,\]
	which, together with \eqref{eq:cqk}, implies
	\[0\le \inprod{z - z^{\infty}}{w}\quad \forall \, z,w \mbox{ satisfying } w\in\cT(z). \]
	From the maximality of $\cT$, we know that $0\in\cT(z^\infty)$.
	Hence, we can replace $\bar z$ in \eqref{eq:barz} by $z^\infty$.
	Therefore,
	\[
	\lim_{k\to \infty}\norm{z^k - z^{\infty}}_{\cM_k} = 0.
	\]
	That is $\lim_{k\to \infty} z^k = z^{\infty}$.
\end{proof}

Next, we study the convergence rate of the {preconditioned} proximal point algorithm. The following error bound assumption associated with $\cT$ is critical to the study of the  convergence rate of the preconditioned PPA.
\begin{assumption}
	\label{assump:errorbound}
	For any $r >0$, there exists $\kappa >0$ such that
	\begin{equation}
	\label{eq:errorbound}
	{\rm dist}(x,\cT^{-1}(0)) \le \kappa \,{\rm dist}(0,\cT(x)) \quad
	\forall\, x\in\cX \mbox{ satisfying } {\rm dist}(x,\cT^{-1}(0)) \le r.
	\end{equation}
\end{assumption}
{In Rockafellar's classic work \cite{rockafellar1976monotone}, the asymptotic Q-superlinear convergence of PPA is established under the assumption that $\cT^{-1}$ is Lipschitz continuous at zero.  Note that the Lipschitz continuity assumption on $\cT^{-1}$ is rather restrictive, since it implicitly implies that {$\cT^{-1}(0)$} is a singleton. In \cite{luque1984asymptotic}, Luque extended Rockafellar's work by considering the following relaxed condition over $\cT$: there exist $\gamma >0$ and $\epsilon >0$ such that
\begin{equation}
\label{eq:luqueER}
{\rm dist}(x,\cT^{-1}(0)) \le \gamma {\rm dist}(0,\cT(x)) \quad
\forall\ x \in \{x\in \cX\mid {\rm dist}(0,\cT(x)) < \epsilon \}.
\end{equation}
We show in the following lemma that this condition in fact implies Assumption \ref{assump:errorbound}. Thus, our Assumption \ref{assump:errorbound} is quite mild and weaker than condition \eqref{eq:luqueER}.
}

\begin{lemma}
	\label{lemma:errorbound}
Let $F$ be a multifunction from $\cX$ to $\cY$ with $F^{-1}(0) \neq \emptyset$. If $F$ {satisfies condition} \eqref{eq:luqueER}, then Assumption \ref{assump:errorbound} holds for $F$, i.e.,
for any $r>0$, there exists $\kappa >0$ such that
	\begin{equation*}
	{\rm dist}(x,F^{-1}(0)) \le \kappa \,{\rm dist}(0,F(x)) \quad
	\forall\, x\in\cX \mbox{ satisfying } {\rm dist}(x,F^{-1}(0)) \le r.
	\end{equation*}
\end{lemma}
\begin{proof}
	Since $F$ {satisfies condition} \eqref{eq:luqueER}, there exist $\varepsilon >0$ and $\kappa_0 \ge 0$ such that if $x\in \cX$ {satisfies} ${\rm dist}(0,F(x)) < \varepsilon$, then
	\[
	{\rm dist}(x, F^{-1}(0)) \le \kappa_0 {\rm dist}(0,F(x)).
	\]
	For any $r>0$ and $x$ satisfying ${\rm dist}(x,F^{-1}(0)) \le r$, if ${\rm dist}(0, F(x)) < \epsilon$, then ${\rm dist}(x, F^{-1}(0)) \le \kappa_0 {\rm dist}(0,F(x))$; otherwise if ${\rm dist}(0, F(x)) \ge \epsilon$, then
	\[
	{\rm dist}(0, F(x)) \ge \epsilon \ge \frac{\epsilon}{r} {\rm dist}(x, F^{-1}(0)),
	\]
	i.e., ${\rm dist}(x, F^{-1}(0)) \le \frac{r}{\epsilon} {\rm dist}(0, F(x))$.
	Therefore, the desired inequality holds for $\kappa = \max\{\kappa_0, \frac{r}{\epsilon}\}$.
\end{proof}
\begin{remark}
	\label{rmk:errorboundpolyhedral}
{In fact, condition \eqref{eq:luqueER} is exactly the local upper Lipschitz continuity of $\cT^{-1}$ at the origin which was introduced by Robinson in \cite{robinson1976implicit}. Later, Robinson established in \cite{robinson1981some} the {celebrated result}  that every polyhedral multifunction is locally upper Lipschitz continuous, i.e., satisfies condition \eqref{eq:luqueER}. {Thus} from Lemma \ref{lemma:errorbound}, we know that any polyhedral multifunction $F$ with $F^{-1}(0) \neq \emptyset$ satisfies Assumption \ref{assump:errorbound}.}
\end{remark}

{
Since the nonnegative sequences $\{ \nu_k \}$ and $\{\epsilon_k \}$  in condition \eqref{eq:condMk} and the stopping criterion (A), respectively, are summable, we know that $0 < \Pi_{k=0}^{\infty} (1+\nu_k)<+\infty$ and we
can choose $r$ to be a positive number satisfying $r > \sum_{k=0}^{\infty} \epsilon_k(1+\nu_k)$. Assume that $\cT$ satisfies Assumption \ref{assump:errorbound}, then associated with $r$, there exists a positive constant $\kappa$ such that \eqref{eq:errorbound} holds. With these preparations, we prove in the following theorem the asymptotic Q-superlinear (R-superlinear) convergence of the weighted (unweighted) distance between the sequence generated by the preconditioned PPA and $\Omega$.
{
\begin{theorem}
	\label{thm:rate}
		Suppose that $\Omega \neq \emptyset$ and the initial {point $z^0$ satisfies} \[{\rm dist}_{\cM_0}(z^0,\Omega) \le \frac{r -\sum_{k=0}^{\infty} \epsilon_k(1 + \nu_k)}{\Pi_{k=0}^\infty (1+\nu_k)}.\] Let $\{z^k\}$ be the infinite sequence generated by the preconditioned PPA under criteria (A) and (B) with $\{c_k\}$ nondecreasing ($c_k \uparrow c_{\infty} \le \infty$). Then for all $k\ge 0$, it holds that
	\begin{equation}
	\label{eq:rate-mPPA}
	{\rm dist}_{\cM_{k+1}}(z^{k+1},\Omega) \le \mu_k {\rm dist}_{\cM_k}(z^k, \Omega),
	\end{equation}
	where $\mu_k = (1+\nu_k)(1-\delta_k)^{-1} \big(\delta_k + (1+\delta_k) \kappa\lambda_{\max}(\cM_k) / \sqrt{c_k^2 + \kappa^2\lambda^2_{\max}(\cM_k)}\, \big)$ and
	\begin{equation}
	\label{eq:assymRate}
	\limsup_{k\to \infty}\mu_k = \mu_{\infty} = \frac{\kappa \lambda_{\infty}}{\sqrt{c_{\infty}^2 + {\kappa^2}\lambda^2_{\infty}}} < 1
	\quad  (\mu_{\infty} = 0 \mbox{ if } c_{\infty} = \infty),
	\end{equation}
	with $\lambda_{\infty}$ given in \eqref{eq:condMk}.
	{In addition}, one has that for all $k\ge 0$,
	\begin{equation}\label{eq:Rrate}
	{\rm dist}(z^{k+1},\Omega)
	\le \frac{\mu_k}{\sqrt{\lambda_{\min}(\cM_{k+1})}} {\rm dist}_{\cM_k}(z^k,\Omega).
	\end{equation}
\end{theorem}
}
\begin{proof}
		From \eqref{eq:distzk} in Theorem \ref{thm:mPPA}, we know that for all $k\ge 0$,
		$ {\rm dist}_{\cM_k}(z^k,\Omega) 
		\le
		\Pi_{k=0}^{\infty}(1 + \nu_k){\rm dist}_{\cM_0}(z^0,\Omega) + \sum_{k=0}^{\infty} \epsilon_k(1+\nu_k) \le r,
		$
		and consequently,
		\[
		{\rm dist}_{\cM_k}(\cP_k({z^k}),\Omega)
		\le \norm{\cP_k(z^k) - \Pi_{\Omega}(z^k)}_{{\cM_k}}	    
		{= \norm{\cP_k(z^k) - \cP_k(\Pi_\Omega(z^k))}_{\cM_k} } 
		\le {\rm dist}_{\cM_k}(z^k,\Omega)\le r \quad \forall k\ge 0.
		\]
	From Proposition \ref{prop:pkqk} (a), we have
	\[
	c_k^{-1}\cM_k\cQ_k(z^k) \in \cT(\cP_k(z^k)),
	\]
	which, together with Assumption \ref{assump:errorbound}, implies that for all $k\ge 0$
	\[
	{\rm dist}(\cP_{k}(z^k),\Omega) \le \kappa c_k^{-1}
	\norm{\cM_k\cQ_k(z^k)}.
	\]
	It further implies that for all $k\ge 0$,
	\[
	\frac{1}{\sqrt{\lambda_{\max}(\cM_k)}} {\rm dist}_{\cM_k}(\cP_k(z^k),\Omega) \le {{\rm dist}(\cP_k(z^k),\Omega)}
	\le \sqrt{\lambda_{\max}(\cM_k)}\kappa c_k^{-1} \norm{\cQ_k(z^k)}_{\cM_k}.
	\]
	Now taking $\bar z = \Pi_{\Omega}(z^k)$, we deduce from \eqref{eq:Qszpz}
	that for all $k\ge 0$,
	\begin{equation}
	\label{eq-Q}
    \begin{aligned}
		\norm{\cQ_k(z^k)}^2_{\cM_k} \le{}& \norm{z^k - \Pi_{\Omega}(z^k)}^2_{{\cM_k}}
	- \norm{\cP_k(z^k) - \Pi_{\Omega}(z^k)}^2_{\cM_k} \\[5pt]
	\le{}& {\rm dist}_{\cM_k}^2(z^k,\Omega) - {\rm dist}_{\cM_k}^2(\cP_k(z^k),\Omega).
	\end{aligned}
	\end{equation}
	Therefore, it holds that
	\begin{equation}
	\label{eq-dist}
	{\rm dist}_{\cM_k}(\cP_k(z^k),\Omega) \le \frac{\kappa\lambda_{\max}(\cM_k)}{\sqrt{c_k^2 + \kappa^2\lambda^2_{\max}(\cM_k)}} {\rm dist}_{\cM_k}(z^k,\Omega)\quad \forall\,k\ge0.
	\end{equation}
		Under stopping criterion (B), we further have for all $k\ge 0$,
	\begin{eqnarray*}
		&& \hspace{-0.7cm}
		\norm{z^{k+1} - \Pi_{\Omega}(\cP_k(z^k))}_{\cM_k}
		\;\le\;  \norm{z^{k+1} - \cP_k(z^k)}_{\cM_k} + \norm{\cP_k(z^k) - \Pi_{\Omega}(\cP_k(z^k))}_{\cM_k} \\[5pt]
		&\le& \delta_k\norm{z^{k+1} - z^k}_{\cM_k} + \norm{\cP_k(z^k) - \Pi_{\Omega}(\cP_k(z^k))}_{\cM_k}\\[5pt]
		&\le& \delta_k\big(\norm{z^{k+1} - \Pi_{\Omega}(\cP_k(z^k))}_{\cM_k}
		+ \norm{z^k - \Pi_{\Omega}(\cP_k(z^k)}_{\cM_k}  \big)
		+ \norm{\cP_k(z^k) - \Pi_{\Omega}(\cP_k(z^k))}_{\cM_k}.
	\end{eqnarray*}
	Thus,
	\begin{eqnarray*}	&& \hspace{-0.7cm}
		(1-\delta_k) \norm{z^{k+1} - \Pi_{\Omega}(\cP_k(z^k))}_{\cM_k}
		\;\le \;  \delta_k\norm{z^k - \Pi_{\Omega}(\cP_k(z^k)}_{\cM_k}
		+ \norm{\cP_k(z^k) - \Pi_{\Omega}(\cP_k(z^k))}_{\cM_k}.
	\end{eqnarray*}
	Now
	\begin{eqnarray*}
		&& \hspace{-0.7cm}
		\delta_k \norm{z^k - \Pi_{\Omega}(\cP_k(z^k)}_{\cM_k}
		\;\le\; \delta_k \norm{\cP_k(z^k)  - \Pi_{\Omega}(\cP_k(z^k)}_{\cM_k} + \delta_k \norm{\cQ_k(z^k)}_{\cM_k}
		\\[5pt]
		&\leq & \delta_k \norm{\cP_k(z^k)  - \Pi_{\Omega}(\cP_k(z^k)}_{\cM_k} + \delta_k
		{\rm dist}_{\cM_k}(z^k,\Omega),
	\end{eqnarray*}
	where the last inequality follows from \eqref{eq-Q}.
	By using the above inequality in the previous one, we get
	\begin{eqnarray*}
		(1-\delta_k) \norm{z^{k+1} - \Pi_{\Omega}(\cP_k(z^k))}_{\cM_k}
		\;\le \;  \delta_k {\rm dist}_{\cM_k}(z^k,\Omega)
		+ (1+\delta_k) {\rm dist}_{\cM_k}(\cP_k(z^k),\Omega).
	\end{eqnarray*}
	Therefore, from the last inequality and \eqref{eq-dist},  it holds that for all $k \ge 0$,
	\begin{align*}
	{\rm dist}_{\cM_{k+1}}(z^{k+1},\Omega) \le &{}(1+\nu_k){\rm dist}_{\cM_k}(z^{k+1},\Omega) \\ \le &{} (1+\nu_k)\norm{z^{k+1} - \Pi_{\Omega}(\cP_k(z^k))}_{\cM_k}
	\le  {\mu_k {\rm dist}_{\cM_k}(z^k,\Omega),}
	\end{align*}
	where $\mu_k = (1+\nu_k)(1-\delta_k)^{-1}\Big(\delta_k +  (1+\delta_k) \kappa\lambda_{\max}(\cM_k) / \sqrt{c_k^2 + \kappa^2\lambda^2_{\max}(\cM_k)}\Big)$\;.
	That is, \eqref{eq:rate-mPPA} holds for all $k\ge 0$.
    Since for all $k\ge 0$, $\cM_k \succeq \lambda_{\min}\cI$, \eqref{eq:assymRate} and $\eqref{eq:Rrate}$ can be obtained through simple calculations.
\end{proof}
\begin{remark}
	\label{rmk:ratepPPA}
	{Suppose that $\{ \delta_k \}$ in criterion (B) is nonincreasing and $\nu_k\equiv 0$ for all $k\ge 0$. Since $\{c_k\}$ is nondecreasing and $\lambda_{\max}(\cM_k)$ is nonincreasing, we know that $\{\mu_k\}$ is nonincrasing.
	Therefore, if one chooses $c_0$ large enough such that $\mu_0 < 1$, then we have $\mu_k\le \mu_0 < 1$ for all $k\ge 0$. The inequality \eqref{eq:rate-mPPA} thus implies the global Q-linear convergence of $\{ {\rm dist}_{\cM_k}(z^k,\Omega) \}$. In addition, \eqref{eq:Rrate} implies that for all $k\ge 0$,
	\[
	{\rm dist}(z^{k+1},\Omega) \le \big({\rm dist}_{\cM_0}(z^0,\Omega)/\sqrt{\lambda_{\min}}\big) \Pi_{i=0}^k \mu_i
	\le
	(\mu_0)^{k+1} \big({\rm dist}_{\cM_0}(z^0,\Omega)/\sqrt{\lambda_{\min}}\big),
	\]
	i.e., $\{
	{\rm dist}(z^k,\Omega)\}$ converges globally R-linearly.}
\end{remark}
}
\section{A semismooth Newton proximal augmented Lagragian method}

{
Note that we can equivalently rewrite problem (D) in the following minimization form:
\begin{equation*}
\mbox{(D)} \quad -\min\Big\{g(y):=\delta_K^* (A^*y - c) - b^Ty \Big\}.
\end{equation*}
Associated with this unconstrained formulation, we write the augmented Lagrangian function following the framework developed in \cite[Examples 11.46 and 11.57]{rockafellar2009variational}. To do so, we first identify (D) with the problem of minimizing $g(y) = \widetilde{g}(y,0)$ over $\R^m$ for
\[
\widetilde g (y,\xi) = -b^T y + \delta_K^*(A^*y - c + \xi) \quad \forall\,(y,\xi)\in\R^m \times \R^n.
\]
Obviously, $\widetilde g$ is jointly convex in $(y,\xi)$. Now, we are able to write down the Lagrangian function $l:\R^m\times \R^n$ through partial dualization as follows:
\[
l(y;x) := \inf_{\xi} \left\{
\widetilde g(y,\xi) - \inprod{x}{\xi}
\right\}
= -b^T y - \inprod{x}{c - A^* y} - \delta_K(x).
\]
Thus, the KKT conditions associated with (P) and (D) are given by
\begin{equation}
\label{eq:kktpd}
-b + Ax = 0, \quad A^*y - c \in \partial \delta_K(x),\quad (x,y)\in\R^n\times \R^m.
\end{equation}
Given $\sigma > 0$, the augmented Lagrangian function corresponding to (D) can be obtained by
\begin{align*}
L_{\sigma}(y;x) :={}& \sup_{s\in\R^n}\Big\{
l(y;s) - \frac{1}{2\sigma}\norm{s-x}^2
\Big\} \\[5pt]
={}& -b^T y  - \inf_{s\in\R^n} \Big\{
\delta_K(s) + \inprod{s}{c - A^* y} + \frac{1}{2\sigma}\norm{s-x}^2
\Big\} \\[5pt]
={}& -b^T y - \inprod{\Pi_K(x - \sigma(c - A^*y))}{c - A^*y} - \frac{1}{2\sigma}\norm{\Pi_K(x - \sigma(c - A^*y)) - x}^2.
\end{align*}
We propose to solve (D) via an inexact proximal augmented Lagrangian method. Our algorithm is named as the semi-smooth Newton inexact proximal augmented Lagrangian ({\sc Snipal}) method
because we will design a semi-smooth Newton method to solve the underlying augmented Lagrangian subproblems. Its template is given as follows.

\begin{algorithm}[H]
	\caption{{\sc Snipal}: Semi-smooth Newton inexact proximal augmented Lagrangian}
	\label{alg:pdp}
	Let $\sigma_0,\sigma_\infty>0$ be given parameters, $\{\tau_k\}_{k=0}^{\infty}$ be a given
	{nonincreasing} sequence such that $\tau_k >0$ for all $k\ge0$. Choose $(x^0,y^0)\in \R^n\times \R^m$. For $k=1,\ldots$, perform the following steps in each iteration.
	\begin{description}
		\item[\bf Step 1.]
		Compute
\begin{equation}
y^{k+1} \approx \mbox{argmin}_{y\in\R^m} \Big\{  L_{\sig_k}(y;x^k)
+ \frac{\tau_k}{2\sig_k}\norm{y-y^k}^2 \Big\}
\label{eq:suby}
\end{equation}
via the semismooth Newton method.
		\item[\bf Step 2.]
		Compute
		$
		x^{k+1} = \Pi_K\big(x^k - \sigma_k(c - A^*y^{k+1})\big).
		$
\item[Step 3.] Update $\sig_{k+1} \uparrow \sig_\infty \leq \infty$.
	\end{description}
\end{algorithm}

Note that different from the classic proximal method of multipliers in \cite{rockafellar1976augmented} with $\tau_k \equiv 1$ for all $k$,  we allow an adaptive choice of the parameter $\tau_k$ in the proximal term $\frac{\tau_k}{2\sigma_k}\norm{y-y^k}^2$ in the inner subproblem \eqref{eq:suby} of Algorithm {\sc Snipal}. Here, the proximal term is added to guarantee the existence of the optimal solution to the inner subproblem \eqref{eq:suby}, and to ensure the positive definiteness of the coefficient matrix of the underlying semi-smooth Newton linear system. {Moreover}, our numerical experience with {\sc Snipal} indicates that having
the additional flexibility of {choosing the parameter $\tau_k$ can help to improve} the practical performance of the algorithm. We  emphasize here that comparing with \cite{rockafellar1976augmented}, our modifications {focus} more on the  computational and implementational aspects.

While the introduction of the parameters $\{\tau_k\}$ brings us more flexibility and some promising numerical advantages, it also makes the convergence analysis {of the algorithm} more challenging. Fortunately, we are able to rigorously characterize the connection between our Algorithm {\sc Snipal} and the preconditioned PPA studied in Section \ref{sec:mPPA}. As one will see in the subsequent text, this connection allows us to conduct a comprehensive convergence analysis for Algorithm {\sc Snipal}. {From the convergence analysis, we also note that $\frac{\tau_k}{2\sigma_k}\norm{y-y^k}^2$ can be replaced by a more general proximal term, i.e., $\frac{1}{2\sigma_k}\norm{y-y^k}_{T_k}^2$ with a symmetric positive definite matrix $T_k$.}

\subsection{Global convergence properties of {\sc Snipal}}
\label{sec:globalCon-snipal}
In this section, we present a comprehensive analysis for the convergence properties of {\sc Snipal}. The global convergence and global linear-rate convergence of {\sc Snipal} are presented as an application of the theory of the preconditioned PPA.

To establish the connection between {\sc Snipal} and the preconditioned PPA, we first introduce some notation. To this end, for $k=0,1,\ldots,$ and any given
$(\bar y, \bar x)\in \R^m\times \R^n$, define the function
\begin{equation}
\label{pAL}
P_k(\bar y, \bar x) :=  \displaystyle\arg \minimax_{y,x}\Big\{l(y,x) + \frac{\tau_k}{2\sigma_k}\|y-\bar y\|^2  - \frac{1}{2\sigma_k}\|x-\bar x\|^2\Big\}.
\end{equation}
Corresponding to the closed proper convex-concave function $l$, we can define the maximal monotone operator $\mathcal{T}_l$ \cite[Corollary  37.5.2]{rockafellar1970convex}, by
\begin{align*}
\mathcal{T}_l(y,x) :={}& \{(y',x') \,|\, (y',-x')\in\partial l(y, x)\} \\[5pt]
={}& \{(y',x')\,|\, y'= -b + Ax, \ x'\in c - A^*y + \partial \delta_K(x) \},
\end{align*}
whose corresponding inverse operator is given by
\begin{equation}\label{Tinverse}
\mathcal{T}^{-1}_l(y',x') := \arg\minimax_{y,x}\,\{l(y,x)-\langle y',y\rangle  + \langle x',x\rangle\}.
\end{equation}
Since $K$ is a polyhedral set,  {$\partial \delta_K$ is known to be a polyhedral multifunction} (see, e.g., \cite[p.~108]{Klatte1995nonsmooth}). As the sum of two polyhedral multifunctions is also polyhedral, $\cT_l$ is also polyhedral.
Define, for $k=0,1,\ldots,$
\begin{equation}\label{eq-Lambda}
\Lambda_k = \text{Diag}\,(\tau_k I_m, I_n)\succ 0.
\end{equation}
The optimal solution of problem \eqref{pAL}, i.e., $P_k({\bar{y}},\bar{x})$, can be obtained via  the following lemma.

\begin{lemma}\label{lemma:ppa}
	For all $k\geq 0$, it holds that
	\begin{equation}\label{PPA_formulate}
	P_k(\bar y,\bar x) = (\Lambda_k + \sigma_k \cT_{l})^{-1}\Lambda_k(\bar y,\bar x)\quad
	\forall\, (\bar y, \bar x)\in\R^m\times\R^n.
	\end{equation}
	If $(y^*,x^*)\in\cT_l^{-1}(0)$, then $P_k(y^*,x^*) = (y^*,x^*)$.
\end{lemma}

In {\sc Snipal}, at $k$-th iteration, denote
\begin{equation}
\label{eq:subpsi}
\psi_k(y): = L_{\sigma_k}(y;x^k) + \frac{\tau_k}{2\sigma_k}\norm{y - y^k}^2.
\end{equation}
From the property of the proximal mapping, we know that
$\psi_k$ is continuously differentiable and
\[
\nabla \psi_k(y) = -b +  A \,\Pi_K\big(x^k + \sigma_k(A^*y - c)\big)
+ {\tau_k}\sig_k^{-1}(y-y^k).
\]
As a generalization of Proposition 8 in  \cite{rockafellar1976augmented}, the following proposition about the weighted distance between $(y^{k+1},x^{k+1})$ generated by {\sc Snipal} and $P_k(y^k,x^k)$ is important for designing the stopping criteria for the subproblem \eqref{eq:suby} and establishing the connection between {\sc Snipal} and
the preconditioned PPA.

\begin{proposition}
	\label{prop:vx-ippa}
	Let  $P_k$, $\Lambda_k$ and $\psi_k$ be defined in  \eqref{pAL}, \eqref{eq-Lambda} and \eqref{eq:subpsi}, respectively. Let $(y^{k+1},x^{k+1})$ be generated by {Algorithm {\sc Snipal}} at iteration $k+1$. It holds that
	\begin{equation}
	\label{eq-err-ippa}
	\norm{ (y^{k+1}, x^{k+1}) - P_k(y^k,x^k)}_{\Lambda_k}
	\le \frac{\sigma_k}{\min(\sqrt{\tau_k},1)}\,
	\norm{\grad \psi_k(y^{k+1})}.
	\end{equation}
\end{proposition}
\begin{proof}
	Since
	$\grad \psi_k(y^{k+1}) = \grad_y L_{\sigma_k}(y^{k+1},x^k) + \tau_k\sigma_k^{-1}(y^{k+1}-y^k),$
  {we have}
	\[\grad \psi_k(y^{k+1}) + \sigma_k^{-1}\tau_k(y^{k}-y^{k+1}) =
	\grad_y L_{\sigma_k}(y^{k+1},x^k),\]
	which, by \cite[Proposition 7]{rockafellar1976augmented}, implies
	$(\grad \psi_k(y^{k+1}) + \sigma_k^{-1}\tau_k(y^{k}-y^{k+1}), \sigma_k^{-1}(x^k - x^{k+1}))\in \cT_l(y^{k+1},x^{k+1}).$
	Thus,
	\[\sigma_k(\grad \psi_k(y^{k+1}),0) + \Lambda_k\big((y^k,x^k) - (y^{k+1},x^{k+1})\big)
	\in \sigma_k \cT_l(y^{k+1},x^{k+1})\]
	and
	$\sigma_k(\grad \psi_k(y^{k+1}),0) + \Lambda_k (y^k,x^k) \in (\Lambda_k +\sigma_k\cT_{l})(y^{k+1},x^{k+1}),$
	or equivalently,
	\[(y^{k+1},x^{k+1}) = (\Lambda_k +\sigma_k\cT_{l})^{-1}\Lambda_k\big(\Lambda_k^{-1}(
	\sigma_k\grad \psi_k(y^{k+1}),0) + (y^k,x^k)\big).\]
	Then, by Lemma \ref{lemma:ppa} and Proposition \ref{prop:pkqk}, we know that
	\begin{equation*}\label{eq-error}
	\begin{aligned}
	&\norm{ (y^{k+1}, x^{k+1}) - P_k(y^k,x^k)}_{\Lambda_k}\\[5pt]
	={}& \norm{(\Lambda_k +\sigma_k{\cT_{l}})^{-1}\Lambda_k\big(\Lambda_k^{-1}(\sigma_k\grad \psi_k(y^{k+1}),0) + ({y^k},x^k)\big) - (\Lambda_k + \sigma_k \cT_{l})^{-1}\Lambda_k\big(({y^k},x^k)\big)}_{\Lambda_k}\\[5pt]
	\le {}& \norm{\Lambda_k^{-1}\big(\sigma_k\grad \psi_k(y^{k+1}),0\big)}_{\Lambda_k} \le {} \frac{\sigma_k}{\min{(\sqrt{\tau_k},1)}}\norm{\grad \psi_k(y^{k+1})}.
	\end{aligned}
	\end{equation*}
	This completes the proof for the proposition.
\end{proof}
}

Based on Proposition \ref{prop:vx-ippa}, we propose the following stopping criteria for the approximate computation of $y^{k+1}$ in Step 1 of {\sc Snipal}:
\begin{align*}
& \mbox{(A')}\quad  \norm{\grad\psi_k(y^{k+1})} \le  \frac{\min(\sqrt{\tau_k},1)}{\sigma_k}\epsilon_k,\quad 0 \le \epsilon_k,  \quad \mbox{$\sum_{k=0}^{\infty}$} \epsilon_k < \infty,
\\[5pt]
&\mbox{(B')}\quad  \norm{\grad\psi_k(y^{k+1})} \le \frac{\delta_k\min(\sqrt{\tau_k},1)}{\sigma_k}\norm{(y^{k+1},x^{k+1}) - (y^k,x^k)}_{\Lambda_k}, \; 0\le \delta_k < 1, \;
\mbox{$\sum_{k=0}^{\infty}$} \delta_k < \infty.
\end{align*}
For the convergence of {\sc Snipal}, we also need the following assumption on $\tau_k$:
\begin{assumption}
	\label{ass:tauk}
The positive sequence $\{\tau_k\}$ is non-increasing and bounded away from zero, i.e., $\tau_k \downarrow \tau_{\infty} > 0$ for some positive constant $\tau_{\infty}$.
\end{assumption}
Under Assumption \ref{ass:tauk}, we have that for all $k\ge 0$,
\[\Lambda_{k}\succeq \Lambda_{k+1} \mbox{ and } \Lambda_{k} \succeq \min(1,\tau_{\infty}) I_{m+n}.\]
We now present the global convergence result for {\sc Snipal} in the following Theorem. Similar to the case in \cite{rockafellar1976augmented}, it is in fact a direct application of Theorem \ref{thm:mPPA}.

\begin{theorem}[Global convergence of {\sc Snipal}]
	\label{thm:snipalconvergence}
Suppose that Assumptions \ref{assump:fesPD} and \ref{ass:tauk} hold. Let
$\{(y^k,x^k)\}$ be the sequence generated by Algorithm
{\sc Snipal} with the stopping criterion (A'). Then $\{(y^k,x^k)\}$ is bounded. In addition, $\{x^k\}$ converges to an optimal solution of (P) and
$\{y^k\}$ converges to an optimal solution of (D), respectively.
\end{theorem}

Since $\cT_l$ is a polyhedral multifunction, we know from Lemma \ref{lemma:errorbound} and Remark \ref{rmk:errorboundpolyhedral} that $\cT_l$ satisfies Assumption \ref{assump:errorbound}.
Let $r$ be a positive number satisfying $r > \sum_{i=0}^{\infty} \epsilon_k$ with $\epsilon_k$ being the summable sequence in (A').
Then, there exists $\kappa >0$ associated with $r$ such that for any $(y,x)\in\R^{m}\times\R^n$ satisfying
${\rm dist}((y,x),{\cT_l^{-1}(0))}\le r,$
\begin{equation}
\label{eq:errorboundTl}
{\rm dist}((y,x),\cT^{-1}_l(0)) \le \kappa\, {\rm dist}(0,\cT_l(y,x)).
\end{equation}
As an application of Theorem \ref{thm:rate}, we are now ready to show the asymptotic superlinear convergence of {\sc Snipal} in the following theorem.
\begin{theorem}[Asymptotic superlinear convergence of {\sc Snipal}]

Suppose that Assumptions \ref{assump:fesPD} and \ref{ass:tauk} hold and the initial $z^0:=(y^0,x^0)$ satisfies ${\rm dist}_{\Lambda_0}(z^0,\cT_l^{-1}(0)) \le r -\sum_{i=0}^{\infty} \epsilon_k$. Let $\kappa$ be the modulus given in \eqref{eq:errorboundTl} and $\{z^k:=(y^k,x^k)\}$ be the infinite sequence generated by the preconditioned {PPA} {under criteria (A') and (B')}. Then, for all $k\ge 0$, it holds that
\begin{equation}
\label{eq:asymSnipal}
\begin{aligned}
&{\rm dist}_{\Lambda_{k+1}}(z^{k+1},\cT_l^{-1}(0)) \le { \mu_k {\rm dist}_{\Lambda_{k}}({z^{k}},\cT_l^{-1}(0))}, \\[5pt]
&{\rm dist}(z^{k+1},\cT_l^{-1}(0)) \le {\frac{\mu_k}{\sqrt{\min(1,\tau_{k+1})}}} {\rm dist}_{\Lambda_k}(z^k,\cT_l^{-1}(0)),
\end{aligned}
\end{equation}
where $\mu_k = (1-\delta_k)^{-1}\Big(
\delta_k + (1+\delta_k)\kappa \gamma_k /\sqrt{\sigma_k^2 + \kappa^2 \gamma_k^2}\Big)$ with $\gamma_k := \max(\tau_k,1)$ and
\begin{align*}
\lim_{k\to \infty} \mu_k = \mu_{\infty} = \frac{\kappa \gamma_{\infty}}{\sqrt{\sigma_{\infty}^2 + {\kappa^2} \gamma_{\infty}^2}} < 1 \quad (\mu_{\infty} = 0 \mbox{ if } \sigma_{\infty} = \infty),
\end{align*}
with $\gamma_{\infty} = \max(\tau_{\infty}, 1)$.
\end{theorem}

\begin{remark}
	\label{rmk:ratepsnipal}
	Suppose that $\{ \delta_k \}$ in criterion (B') is nonincreasing.
We know	from Remark \ref{rmk:ratepPPA} that if one chooses $\sigma_0$ large enough such that $\mu_0 < 1$, then $\mu_k\le \mu_0 < 1$ for all $k\ge 0$. Thus, from \eqref{eq:asymSnipal}, we have the global linear convergence of $\{{\rm dist}_{\Lambda_{k}}(z^k, \cT_l^{-1}(0)) \}$
and $\{
{\rm dist}(z^k,\cT_{l}^{-1}(0))
\}$.
\end{remark}

\subsection{Semismooth Newton method for subproblems \eqref{eq:suby}}
In this subsection, we discuss how the subproblem \eqref{eq:suby} in {\sc Snipal} can be solved efficiently. As is mentioned in the name of {\sc Snipal}, we propose to solve \eqref{eq:suby} via an inexact semismooth Newton method which converges at least {locally}
superlinearly. In fact, the {local} convergence rate can even be quadratic. 

For given $(\tilde x, \tilde y)\in\R^n\times \R^m$ and $\tau, \sigma >0$, define {the} function $\psi:\R^m \to \R$ as
\[
\psi(y) := L_{\sigma}(y;\tilde x) +
\frac{\tau}{2\sigma}\norm{ y - \tilde y}^2 \quad \forall y\in\R^m,
\]
and
we aim to solve
\begin{equation}\label{eq:probpsi}
\min_{y\in\R^m} \psi(y).
\end{equation}
Note that $\psi$ is strongly convex and continuously differentiable over $\R^m$ with
\[
\nabla \psi(y) = -b +  A \,\Pi_K\big(\tilde x + \sigma(A^*y - c)\big)
+ {\tau}\sig^{-1}(y-\tilde y).
\]
Hence, we know that for any given {$\alpha \ge \inf_y \psi(y)$}, the level set $\cL_{\alpha}:= \{y\in \R^m \mid\psi(y)\le \alpha \}$ is a nonempty closed and bounded convex set. In addition, problem \eqref{eq:probpsi} has a unique optimal solution
{which we denote} as $\bar y$.

As an unconstrained optimization problem, the optimality condition for \eqref{eq:probpsi} {is given by}
\begin{equation}
\label{eq:nonsmootheq}
\nabla \psi(y) = 0, \quad y\in \R^m,
\end{equation}
and $\bar y$ is the unique solution to this nonsmooth equation. Since $\Pi_K$ is a {Lipschitz} continuous piecewise affine function, we have that
$\nabla \psi$ is strongly semismooth. Hence, we can solve the nonsmooth equation \eqref{eq:nonsmootheq} via a {semismooth} Newton method. For this purpose, we define the following operator:
\[
\hat\partial^2\psi(y): = \tau\sigma^{-1} I_{m} + \sigma A\partial \Pi_K(\tilde x + \sigma(A^* y - c))A^* \quad
\forall y\in\R^m,
\]
where $\partial \Pi_{K}(\tilde x + \sigma(A^*y - c))$ is the Clarke subdifferential \cite{Clarke1983Optimization} of the Lipschitz continuous mapping $\Pi_K(\cdot)$ at $\tilde x + \sigma(A^*y - c)$.
Note that from \cite[Example 2.5]{Hiriart1984geeralized}, we have
that
\[
\hat\partial^2\psi(y) d = \partial^2\psi(y) d \quad \forall d\in\R^m,
\]
where $\partial^2\psi(y)$ denotes the generalized Hessian of $\psi$ at $y$. However, {we caution the reader that
it is} unclear whether $\hat\partial^2\psi(y) = \partial^2\psi(y)$.
 Given any $y\in \R^m$, define
\begin{equation}
\label{eq:genHessian}
H:= \tau\sigma^{-1}I_m + \sigma AUA^*
\end{equation}
with $U\in \partial \Pi_K(\tilde x + \sigma (A^*y - c))$.
Then, we know that $H\in\hat\partial^2\psi(y)$ and $H$ is symmetric positive definite.

After these preparations, we are ready to present the following semismooth Newton method for solving the nonsmooth equation \eqref{eq:nonsmootheq} and {we can} expect a fast {local} superlinear convergence.
\begin{algorithm}[H]
	\caption{{\sc Ssn}: A Semi-smooth Newton method for  solving \eqref{eq:nonsmootheq}
	({\sc Ssn}($\tilde{x},\tilde{y},\sigma,\tau$))}
	\label{alg:ssn}
	Given {$\tau>0, \sigma >0$}, choose parameters $\bar\eta\in(0,1),\gamma\in(0,1]$ and $\mu\in(0,1/2), \delta\in(0,1)$ and {set $y^0=\tilde{y}$}.	
	Iterate the following steps for $j = 0,1,\ldots$.	\begin{description}
		\item[\bf Step 1.]
	Choose $U_j\in \partial\Pi_K(\tilde x + \sigma (A^*y^j - c))$. Set
	$H_j: = \tau\sigma^{-1}I_{m} + \sigma AU_jA^*$. Solve the linear system
	\begin{equation}
	\label{eq:ssn-ls}
	H_j d = -\nabla \psi(y^j)
	\end{equation}
	exactly or by {a Krylov   iterative method} to find $d^j$ such  that
	$\norm{H_j d^j + \nabla \psi(y^j)} \le \min (\bar\eta, \norm{\nabla\psi(y^{j})}^{1+\gamma}). $
		\item[\bf Step 2.]
		(Line search)
		Set $\alpha_j = \delta^{m_j}$, where $m_j$ is the  first nonnegative integer $m$ for which
		\[
		\psi(y^j + \delta^m d^j) \le \psi(y^j) + \mu \delta^m\inprod{\nabla \psi(y^j)}{d^j}.
		\]
		\item[Step 3.] Set $y^{j+1} = y^j + \alpha_j d^j$.
	\end{description}
\end{algorithm}
The convergence results of Algorithm {\sc Ssn} are stated in the following theorem.
\begin{theorem}
	\label{thm:convergence_SSN}
Let $\{ y^j\}$ be the infinite sequence generated by Algorithm {\sc Ssn}. It holds that $\{y^j\}$ converges to the unique optimal solution $\bar y$ of \eqref{eq:probpsi} and $\norm{y^{j+1} - \bar y} = \cO(\norm{y^j - \bar y}^{1+\gamma})$.
\end{theorem}
\begin{proof}
We know from \cite[Proposition 3.3]{Zhao2010Newton} that $d^j$ is always a descent direction. Then, the strong convexity of $\psi$ and \cite[Theorem 3.4]{Zhao2010Newton} imply that $\{y^j\}$ converges to the unique optimal solution $\bar y$ of \eqref{eq:probpsi}. By \eqref{eq:genHessian}, we have that {the symmetric positive definite matrix $H_j \in \hat \partial^2 \psi(y^j)$ satisfies the property that $H_j \succeq \tau\sig^{-1}I_m$ for all $j$}. The desired results thus can be obtained by following the proof of \cite[Theorem 3.5]{Zhao2010Newton}. We omit the details here.
\end{proof}

\subsection{ Finite termination property of {\sc Snipal}}
In our numerical experience with {\sc Snipal}, we observe that {it} {nearly possesses a} certain finite convergence property for solving (P) and (D) when $\sigma_k$ and $1/\tau_k$ are sufficiently large. {We note that most available theoretical results corresponding to the finite termination property of proximal point algorithms require each subproblem involved to be solved exactly, e.g., see \cite{rockafellar1976monotone}, \cite{rockafellar1976augmented} and \cite{luque1984asymptotic}.
Hence, all these results cannot be directly adopted to support our numerical findings. In this section, we aim to
investigate the
finite termination property of {\sc Snipal} by showing that it is possible to obtain a solution pair of (P) and (D) without requiring the exact solutions of each and every subproblem involved in {the algorithm}.}

Our {analysis is based on} an interesting property called ``{\em staircase} property'' associated with subdifferential mappings of convex closed polyhedral functions.
 Let \[ f(x):= c^Tx + \delta_K(x) + \delta_{\{x\mid Ax = b\}}(x).\] 
 Clearly, $f$ is a convex closed polyhedral function. From
\cite[Sec. 6]{Eckstein1992Douglas} and earlier work in
 \cite{Durier1988locally,luque1984asymptotic}, we know that its subdifferential mapping enjoys the following  {\em staircase} property, i.e., there exists $\delta > 0$ such that
\begin{equation}\label{eq:staircase}
	w \in \partial f(x), \norm{w} \le \delta \Rightarrow 0 \in \partial f(x).
\end{equation}
Based on the {\em staircase} property of $\partial f$, we present the finite convergence property of {\sc Snipal} in the following theorem.
\begin{theorem}
	\label{thm:finite}
Suppose that {Assumptions \ref{assump:fesPD} and \ref{ass:tauk} hold}
and let $\{(y^l,x^l)\}$ be the infinite sequence generated by {\sc Snipal} with the stopping criterion ${\rm (A')}$. {For any given $k \ge 0$,
suppose that
$\bar{y}^{k+1}$ is an exact
solution to the following optimization problem:
\begin{eqnarray}
\bar{y}^{k+1} = \argmin_{ {y\in\R^m}} L_{\sigma_k}(y;x^k).
\label{eq-11}
\end{eqnarray}
}
{Then, the following results hold.}
\\[5pt]
(a)
The point $\bar{x}^{k+1} := \Pi_K\big(x^k - \sigma_k(c - A^*\bar y^{k+1})\big)$
is the unique solution to the following proximal problem:
\begin{eqnarray}
 \min\Big\{ c^T x +\frac{1}{2\sig_k}\norm{x-x^k}^2 \mid Ax = b, x\in K \Big\}.
 \label{eq-ppa-sub}
\end{eqnarray}
(b) There exists a positive scalar $\bar{\sig}$ independent of $k$
such that for all $\sig_k \geq \bar{\sig}$,
$\bar{x}^{k+1}$ also solves the problem (P).
\\[5pt]
(c) If  $x^k$ is a solution of (P),
then $\bar{y}^{k+1}$ also solves (D).
\end{theorem}
\begin{proof}
(a) Observe that  the dual of \eqref{eq-11} is exactly
\eqref{eq-ppa-sub}, and
the KKT conditions associated with \eqref{eq-11} and \eqref{eq-ppa-sub}
are given as follows:
\begin{eqnarray}
 x =\Pi_{K}\big(x^k - \sigma_k(c - A^* y)\big), \quad A x -b = 0.
 \label{eq-KKT-sub}
\end{eqnarray}
Since  $\bar{y}^{k+1}$ is a solution of the {problem} \eqref{eq-11}, it holds from the optimality condition associated with \eqref{eq-11} that $A\Pi_K(x^k - \sigma_k(c - A^*\bar y^{k+1})) = b$.
Thus, $(\bar x^{k+1}, \bar y^{k+1})$ satisfy \eqref{eq-KKT-sub}.
Therefore, $\bar x^{k+1}$ solves \eqref{eq-ppa-sub}. The uniqueness {of
$\bar{x}^{k+1}$} follows directly from the strong convexity of \eqref{eq-ppa-sub}.

(b) By Theorem \ref{thm:snipalconvergence}, we know that $x^l \to x^*$ as $l\to \infty$ for some $x^*\in \partial f^{-1}(0)$. Therefore, there exists a constant $M>0$ (independent of $k$) such that
\begin{equation}
\label{eq:xlbounded}
\norm{x^l - x^*}\le M \quad \forall\ l\ge 0.
\end{equation}
From the optimality of $\bar{x}^{k+1}$ and the definition of  $f$, we have that
\[\frac{1}{\sigma_k}(x^{k} - \bar{x}^{k+1}) \in \partial f(\bar x^{k+1}). \]
It also holds from the nonexpansive property of the proximal mapping that $
\norm{\bar x^{k+1} - x^*} \le \norm{x^k - x^*},
$
which, together with \eqref{eq:xlbounded}, further implies that
\[\norm{\bar x^{k+1} - x^k} \le 2\norm{x^k - x^*} \le 2M.\]
Therefore, there exists $\bar \sigma>0$ (independent of $k$) such that
for all $\sigma_k \ge \bar \sigma $ and $k\ge 0$,
\[ \frac{1}{\sigma_k}\norm{\bar{x}^{k+1} - x^k} \le \frac{2M}{\bar \sigma} \le \delta, \]
where $\delta>0$ is the constant given in \eqref{eq:staircase}.
Thus, by using the ``staircase'' property \eqref{eq:staircase}, we know that
\[ 0 \in \partial f(\bar{x}^{k+1}).\]
That is, $\bar{x}^{k+1}$ solves the problem (P).

(c) Next, consider the case when $x^k$ is a solution of (P).
From the minimization property of $x^k$,
it is clear that
the unique solution of \eqref{eq-ppa-sub} must be $\bar{x}^{k+1}=x^k.$
Thus, $x^k = \Pi_K\big( x^k - \sigma_k(c-A^*\bar y^{k+1}) \big)$ and $Ax^k = b$.
Note that it can be equivalently rewritten as:
\[
A^*\bar y^{k+1} - c \in \partial\delta_K(x^k), \quad Ax^k = b,
\]
i.e., $({ x^k}, \bar y^{k+1})$ satisfy the KKT conditions for (P) and (D) in \eqref{eq:kktpd}. Thus, $\bar y^{k+1}$ solves (D).
\end{proof}

\begin{remark} We now remark on the significance of the above theorem.
{Essentially, it says that when $\sig_k$ is sufficiently large with $\sig_k \geq \bar{\sig}$,  then $\bar{x}^{k+1}$ solves (P), and
it holds that
$\bar{y}^{k+2} = \argmin L_{\sigma_{k+1}}(y;\bar{x}^{k+1})$ solves (D).}

From the fact that the {\sc Ssn} method used to {solve}
\eqref{eq:nonsmootheq} has the finite termination property \cite{Fischer1996finite,sun1998finite}, we know that
$y^{k+1}$ computed in Step 1 of {\sc Snipal} is in fact the exact solution of the subproblem $\min \psi_k(y)$
when {the corresponding linear system is solved exactly.} 
In addition, when  $\sigma_k$ is sufficiently large and $\tau_k$ is small enough, we have that
\[
{0 = \; \nabla L_{\sigma_k}(y^{k+1};x^k) + \tau_k\sigma_k^{-1}(y^{k+1} - y^k) \approx \nabla L_{\sigma_k}(y^{k+1};x^k),}
\]
and consequently, $y^{k+1}$ can be regarded as a highly accurate solution to the problem
$\min L_{\sigma_k}(y;x^k)$.
In this sense, Theorem \ref{thm:finite} explains the finite termination phenomenon in the practical performance of {\sc Snipal}.
\end{remark}

\section{Solving the linear systems arising from the semismooth Newton method}

Note that the most expensive operation in Algorithm {\sc Ssn} is the computation of the search direction $d\in\R^m$ {through} solving the linear system \eqref{eq:ssn-ls}. To ensure the efficiency of {\sc Ssn} and consequently that of {\sc Snipal}, in this section, we shall discuss efficient approaches for  {solving} \eqref{eq:ssn-ls} in Algorithm {\sc Ssn}. Given $c,\tilde x\in \R^n$, $\tilde y\in\R^m$, the parameters $\tau, \sigma >0$ and the current iterate of {\sc Ssn} $\hat y\in\R^m$, let
$$
 g := - \grad \psi(\hat{y}) =  R_p - \tau \sig^{-1} (\hat{y}-\tilde y),
$$
where $R_p = b - A \Pi_K(w(\hat y))$ with $w(\hat y): = \tilde x + \sigma(A^*\hat y - c)$.
At each {\sc Ssn} iteration, we need to solve a linear system of the form:
\begin{eqnarray}
 H \Delta y \;=\; g,
 \label{eq-Newton}
\end{eqnarray}
where $H = \tau \sig^{-1} I_m + \sigma AUA^*$ with
$U\in \partial \Pi_{K}(w(\hat y))$. Define the index set $\cJ = \{ i \mid l_i <  [w(\hat{y})]_i < u_i, i=1,\ldots,n\}$
and $p = |\cJ|$, i.e., the cardinality of $\cJ$. In the implementation, we always construct the generalized Jacobian matrix $U \in \partial \Pi_{K}(w(\hat y))$ as a diagonal matrix in the following manner:
\begin{equation*}
U = {\rm Diag}(u) \mbox{ with } u_i = \left\{
\begin{aligned}
& 1 \quad \mbox{if} \quad i\in \cJ, \\
& 0 \quad \mbox{otherwise,}
\end{aligned}
\right.
\quad i = 1,\ldots,n.
\end{equation*}
{Without} the loss of generality, we can partition
$A \equiv [A_\cJ, A_\cN]$ with $A_\cJ\in \R^{m\times p}$, $A_\cN\in\R^{m\times(n-p)}$, and
hence
\begin{eqnarray}
H =\sig A_\cJ A_\cJ^* + \tau \sig^{-1} I_m = \sig ( A_\cJ A_\cJ^*+\rho I_m)
\label{eq-H}
\end{eqnarray}
where $\rho := \tau \sig^{-2}$.
To solve the linear system \eqref{eq-Newton} efficiently, we need to consider various scenarios.
In the discussion below, we use ${\tt nnzden}(M)$ to denote the density of the nonzero elements
of a given matrix $M$.

(a) First, we consider the case where $p\geq m$ and the sparse Cholesky factorization of $A_\cJ A_\cJ^*$ can be
computed at a moderate cost. In this case, the main cost of solving the linear system
is in forming the matrix $A_\cJ A_\cJ^*$ at the cost of $O(m^2p\, {\tt nnzden}(A_\cJ))$
and computing the sparse Cholesky factorization of $A_\cJ A_\cJ^* + \rho I_m$.

{Observe that the index set {$\cJ$} generally changes from one SSN iteration to the next. However,
when the SSN method is converging, the index set $\cJ$ may only {change} slightly from the current iteration
to the next. In this case,  one can update the inverse of $H$ {via} a low-rank update by using the
Sherman-Morrison-Woodbury formula.}

When it is expensive to compute and factorize $H$, one would naturally use a preconditioned
conjugate gradient (PCG) method  {or the MINRES (minimim residual) method} to
solve \eqref{eq-Newton}.
Observe that the condition number of $H$ is given by
{$\kappa(H) =  ( \omega_{\max}^2 + \rho)/(\omega_{\min}^2 + \rho)$}
 if $p\geq m$, where
{$\ome_{\max},\ome_{\min}$} are the largest and smallest singular value of
$A_\cJ$, respectively.
{Note that when $A$ is not explicitly given as a matrix, one can compute 
the matrix-vector product $Hv$ as follows: $Hv = \sigma \rho v + \sigma A (e_\cJ \circ (A^* v))$ where
$e_\cJ\in\R^n$ is a $0$-$1$ vector whose non-zero entries are located at the index set $\cJ$,
and ``$\circ$'' denotes the elementwise product.
}

(b) Next we consider the case where  $p < m$. In this case, it is more economical to solve
\eqref{eq-Newton}
by using the Sherman-Morrison-Woodbury formula to get
\begin{eqnarray}
 \Delta y = H^{-1} g = \tau^{-1} \sig\big( I_m - P_\cJ \big) g,
\label{eq-dy}
\end{eqnarray}
where $P_\cJ =  A_\cJ G^{-1} A_\cJ^*$,
$G =\rho I_p +  A_\cJ^* A_\cJ \in \R^{p\times p}$.
Thus to compute $\Delta y$, one needs only to solve a smaller $p\times p$ linear system of equations
$Gv = A_\cJ^*g$.
Observe that  when $\rho \ll 1$,  $ \Delta y$ is approximately the orthogonal projection of $ \tau^{-1} \sigma g$
onto the null space of $A_\cJ^*$.

To solve \eqref{eq-dy}, one can compute the sparse Cholesky factorization of
the symmetric positive definite matrix $G\in \R^{p\times p}$ if the task can be done at a reasonable cost.
In this case, the main cost involved in \eqref{eq-dy} is in computing $A_\cJ^*A_\cJ$  at the cost
of $O(p^2 m\, {\tt nnzden}(A_\cJ))$ operations and the sparse Cholesky factorization of
$G = \rho I_p + A_\cJ^*A_\cJ.$

When it is too expensive to compute and factorize $G$,
one can use a Krylov iterative method to solve the $p\times p$ linear system of equations:
\begin{equation}\label{eq:Gv}
G v = (\rho I_p + A_\cJ^* A_\cJ) v =  A_\cJ^*g .
\end{equation}
To estimate the convergence rate of the
 Krylov iterative method, it is important for us to analyse the
conditioning of the above linear system, as is done in the next theorem.

\begin{theorem} Let $B\in \R^{m\times p}$ with $p<m$. Consider
linear system $G v = B^* g$, where $G=B^*B + \rho I_p$ and $g\in \R^m$.
Then the effective condition number for solving the system by the MINRES (minimum residual) method with zero initial point
is given by
\begin{eqnarray*}
 \kappa = { \frac{\ome_{\max}^2 + \rho}{\ome_{\min}^2+\rho}}
\end{eqnarray*}
where $\ome_{\max}$ is the largest singular value and $\ome_{\min} > 0$ is
the smallest positive singular value of $B$, respectively.
\end{theorem}
\begin{proof} Consider the following full SVD of $B$:
\begin{eqnarray*}
 B = U \Sigma V^T = [U_1, U_2] \left[ \begin{array}{cc} \hat{\Sigma} & 0 \\ 0 & 0 \end{array}\right]
\left[\begin{array}{c} V_1^T \\ V_2^T \end{array} \right],
\end{eqnarray*}
where $\hat{\Sigma}$ is the diagonal matrix consisting of the
positive singular values of $B$.
Let $\P_k^0$ be the set of polynomials $p_k$ with degree at most $k$  and $p_k(0)=1$.
Then for $p_k\in \P_k^0$, we have that
\begin{eqnarray*}
 p_k(G) B^*g &=& V p_k(\Sigma^T\Sigma + \rho I)\Sigma^T U^T g
= [V_1, V_2] \left[ \begin{array}{cc} p_k(\hat{\Sigma}^2 + \rho I) \hat{\Sigma} & 0
\\ 0 & 0 \end{array}
\right] \left[\begin{array}{c} U_1^Tg \\ U_2^Tg \end{array} \right]
\\[5pt]
&=& V_1 p_k(\hat{\Sigma}^2 + \rho I) \hat{\Sigma} U_1^T g.
\end{eqnarray*}
Since the $k$-th iteration of the MINRES method computes an approximate solution $x_{k}$ such that
its residual $\xi = \bar{p}_k(G)B^*g$
satisfies the condition that
\begin{eqnarray*}
 {\norm{\xi}} = \norm{\bar{p}_k(G) B^*g} \;=\; \min_{p_k\in \P_k^0} \, \norm{p_k(G)B^*g}
\;\leq \; {\norm{\hat{\Sigma}U_1^T g}} \, \min_{p_k\in \P_k^0} \,  \norm{p_k(z)}_{[\ome_{\min}^2+\rho,\,\ome_{\max}^2+\rho]} ,
\end{eqnarray*}
thus we see that the convergence rate of the MINRES method
 is determined by the best approximation of the zero function by the polynomials
in $P_k^0$ over the interval
$[\ome_{\min}^2+\rho,\ome_{\max}^2+\rho]$.  
{More specifically, by \cite[Theorem 6.4]{Saad}, we have that 
$
\min_{p_k\in \P_k^0} \,  \norm{p_k(z)}_{[\ome_{\min}^2+\rho,\,\ome_{\max}^2+\rho]} \leq 2 \kappa^{-k}
$.}
Hence the convergence rate of the MINRES method is determined by $\kappa.$
\end{proof}

{After \eqref{eq:Gv} is solved via the MINRES {method}, one can compute the residual {vector} associated with system \eqref{eq-dy} without much difficulty.} Indeed, let the computed solution of \eqref{eq-dy} be given as follows:
\begin{eqnarray*}
 \widehat{\Delta y} = \tau^{-1} \sig (g - A_\cJ v)
\end{eqnarray*}
where $ G v = A_\cJ^* g - \xi$ with $\xi$ being the residual vector obtained from
the MINRES iteration.
Now the residual vector associated with \eqref{eq-dy} is given by
\begin{eqnarray*}
 \eta &:= & g - \cH \widehat{\Delta y} = g - \tau^{-1}\sig \cH g
+\tau^{-1} \sig \cH A_\cJ G^{-1} (A_\cJ^* g - \xi)
\\[5pt]
&=&  g - \tau^{-1}\sig \cH (g - P_\cJ g)
  -\tau^{-1}\sig \cH A_\cJ G^{-1}\xi
\\[5pt]
&=&   -\tau^{-1}\sig \cH A_\cJ G^{-1} \xi
\;=\; - \rho^{-1} A_\cJ \xi,
\end{eqnarray*}
 where the last equation follows directly from the fact that $\cH A_\cJ = \sig A_\cJ G$.
 Based on the computed $\eta$, one can
 check the termination condition for solving the linear system
 in \eqref{eq:ssn-ls}.

{Now, we are ready to bound the condition numbers of the Newton linear systems involved in {\sc Snipal}.
As can be observed from the above discussions, for both cases (a) and (b), the effective condition number of the linear system  involved is upper bounded by
\[
\kappa
\le 1+ \frac{\ome_{\max}^2 }{\rho},
\]	
where {$\ome_{\max}$} is the largest singular value of $A_{\cJ}$ and $\rho = \tau\sigma^{-2}$.
Since $A_\cJ$ is a sub-matrix of $A$, it holds that ${\ome_{\max} \le \norm{A}_2}$. Hence, for any linear {system} involved in the $k$-th iteration of {\sc Snipal}, we can provide an upper bound for the condition number as follows:
\begin{equation}
\label{eq:kappak}
\kappa \le 1 + \frac{\norm{A}_2^2 \sigma_{k}^2}{\tau_k}.
\end{equation}
From our assumptions on {\sc Snipal}, we note that $\sigma_k \le \sigma_{\infty}$ and $\tau_k \ge \tau_{\infty} >0$ for all $k\ge 0$.
Hence, for all the linear systems involved in {\sc Snipal}, there exists an uniform upper bound for the corresponding condition number:
\[
\kappa \le 1 + \frac{\norm{A}_2^2\sigma_{\infty}^2}{\tau_{\infty}}.
\]
As long as $\sigma_{\infty} < +\infty$, we have shown that all these linear systems have bounded condition numbers. This differs significantly from the {setting in}  interior-point based algorithms where the condition numbers of the corresponding normal equations are {asymptotically} unbounded.
The competitive advantage of {\sc Snipal} can be partially explained from the above observation. Meanwhile, in
{the} $k$-th iteration of {\sc Sinpal}, to maintain a small condition number based on  \eqref{eq:kappak}, one should choose small $\sigma_k$ but large $\tau_k$.
However, the convergence rate of {\sc Snipal} developed in Theorem \ref{thm:rate} {requires the opposite choice}, i.e., large $\sigma_k$  {and  $\tau_k$ should be moderate}. {The preceding discussion thus reveals the trade-off between the convergence rate of the ALM and the condition numbers of the Newton linear systems}. Clearly, in the implementation of {\sc Snipal}, {the} parameters
$\{\sigma_k\}$ and $\{\tau_k\}$ should be chosen to balance the progress of the outer and inner algorithms, i.e., the ALM and the semismooth Newton method.
}

\section{Warm-start algorithm for {\sc Snipal}}
As is mentioned in the introduction, to achieve high performance,
it is desirable to use a simple first-order algorithm to warm start {\sc Snipal}
{so that its local linear convergence behavior can be observed earlier}. For this purpose, we present an ADMM algorithm
 for solving (D). We note that {a} similar strategy has also been employed for solving large scale semidefinite programming and quadratic semidefinite programming problems \cite{Yang2015SDPNAL+,LST-QSDPNAL}.

 We begin by rewriting (D) into the following equivalent form:
 \begin{equation}\label{eq:d-admm}
 \min \left\{ \delta_K^*(-z) - b^T y \mid z + A^*y = c  \right\}.
 \end{equation}
 Given $\sigma>0$, the augmented Lagrangian function associated with
 \eqref{eq:d-admm} can be written as
 \[
 {\bf L}_{\sigma}(z,y;x) = \delta_K^*(-z) - b^T y + \inprod{x}{z + A^*y - c} + \frac{\sigma}{2}\norm{z + A^*y - c}^2
 \]
 for all $(x,y,z)\in \R^{n}\times \R^{m}\times \R^n$.
 The template of the classical ADMM for solving \eqref{eq:d-admm} is given as follows.
\begin{algorithm}[H]
	\caption{{\sc ADMM}: An ADMM method for solving \eqref{eq:d-admm}}
	\label{alg:admm}
	Given $(x^0,y^0)\in\R^n\times \R^m$ and $\gamma > 0$, %
     perform the following steps for $k=1,\ldots$,
	\begin{description}
		\item[\bf Step 1.]
		Compute
		\begin{equation}\label{eq:z}
		z^{k+1} = \argmin
		{\bf L}_{\sigma}(z,y^k;x^k)
		\;=\; {\frac{1}{\sigma}\left( \Pi_{K}(x^k + \sigma(A^*y^k - c)) - (x^k + \sigma( A^*y^k - c)) \right).}
		\end{equation}
		\item [\bf Step 2.]
		Compute
		\begin{equation}\label{eq:y}
		y^{k+1} = \argmin
		{\bf L}_{\sigma}(z^{k+1},y;x^k)
		\;=\; {(AA^*)^{-1}\Big( b/\sigma - A(x^k/\sigma + z^{k+1} - c)\Big).}
		\end{equation}
		\item[\bf Step 3.]
		Compute
		$
		x^{k+1} = x^k + \gamma \sigma (z^{k+1} + A^*y^{k+1} - c).$
	\end{description}
\end{algorithm}
The convergence of the above classical ADMM for solving the two-block optimization problem \eqref{eq:d-admm} with the steplength
 {$\gamma \in (0,(1 + \sqrt{5})/2)$} can be readily obtained from the vast literature on ADMM. Here, we adopt a newly developed result from \cite{chen2018equivalent} stating that the above ADMM is in fact an inexact proximal ALM. This %
 {new}
 interpretation allows us to choose the steplength $\gamma$ in the larger interval $(0,2)$ which usually leads a better numerical performance
 {when $\gamma$ is chosen to be $1.9$ instead of $1.618$}. We summarize the convergence results in the following theorem. The detailed proof %
can be found in \cite{chen2018equivalent}.
\begin{theorem}
	\label{thm:ADMM}
	Suppose that Assumption \ref{assump:fesPD} holds and {$\gamma \in (0,2)$}. Let $\{(x^k,y^k,z^k)\}$ be the sequence generated by Algorithm ADMM. Then, $\{x^k\}$ converges to an optimal solution of (P) and $\{(y^k,z^k)\}$ converges to an optimal solution of
	\eqref{eq:d-admm}, respectively.	
\end{theorem}
\begin{remark}
	In the above algorithm, one can also handle \eqref{eq:y} by adding an appropriate proximal term or by using {an iterative method} to solve the corresponding linear system. The convergence of the resulting proximal or inexact ADMM with 
	steplength {$\gamma\in (0,2)$} has also been discussed in \cite{chen2018equivalent}. For simplicity, we only discussed the exact version here.
\end{remark}
\section{Numerical experiments}

In this section, we evaluate the performance of {\sc Snipal} against
the powerful commercial solver Gurobi (version 8.0.1) on various LP data sets.
Our goal is to compare the performance of
our algorithm against {the barrier method implemented in} Gurobi in terms of its speed and
ability to solve the tested instances to the relatively
high accuracy of $10^{-6}$ or $10^{-8}$ in the relative KKT residual. {That is, 
for a given computed solution $(x,y,z)$, we stop the algorithm when}
\begin{eqnarray}
\eta = \max\Big\{ \frac{\norm{b-Ax}}{1+\norm{b}},\,
\frac{\norm{A^T y + z - c}}{1+\norm{c}}, \,
\frac{\norm{x- \Pi_K(x-z)}}{1+\norm{x}+\norm{z}}
\Big\} \; \leq \; {\tt Tol}
\end{eqnarray}
where {\tt Tol} is a given accuracy tolerance. 
{We should note that it is possible to solve an LP by using the 
primal or dual simplex methods in Gurobi, and those methods could 
sometimes be more efficient than the barrier method in solving large scale LPs.
However, as our {\sc Snipal} algorithm is akin to a barrier method, in that
each of its semismooth Newton iteration also requires the solution of a linear system having the 
form of a normal equation just as in the case of the barrier method, we thus restrict
the comparison of our algorithm only to the barrier method in Gurobi.
To purely use the barrier method in Gurobi, we also turn off its crossover capability from the barrier method to simplex methods.}
We should note that sometimes the presolve phrase in Gurobi is
too time consuming and does not lead to any reduction in the
problem size. In that case, we turn off the presolve phase in Gurobi
to {get the actual performance of its barrier method}.

All the numerical experiments in this paper are run in 
{\sc Matlab} on a Dell Laptop with Intel(R) Core i7-6820HQ CPU @2.70GHz and 16GB of RAM. 
{As Gurobi is extremely powerful in exploiting multi-thread computing, we set the number of 
threads allowed for Gurobi to be two so that its overall CPU utilization rate is roughly the same as that observed 
for running {\sc Snipal} in {\sc Matlab} when setting the maximum number of computational threads to be two.
}

\subsection{Randomly generated sparse LP in \cite{Mangasarian-04}}

Here we test  large synthetic LP problems generated
as in \cite{Mangasarian-04}. In particular, the matrix $A$ is generated as follows:
\begin{verbatim}
    rng(`default');  A = sprand(m,n,d);  A = 100*(A-0.5*spones(A));
\end{verbatim}
In this case, we turn off the presolve phase in Gurobi
as this phase is too time consuming for these randomly generated problems and
{it also does not lead to any reduction in the problem sizes.}
As we can observe from Table \ref{table-random},
{\sc Snipal} is able to outperform Gurobi by a factor of about {$1.5-2.3$} times in
computational time in most cases.

{Note that for
 the column ``iter (itssn)'' in Table \ref{table-random}, we report the number of
 {\sc Snipal} iterations and
the total number of semismooth Newton linear systems solved in Algorithm 2. {For the columns ``time (RAM)'' and ``Gurobi time (RAM)'', we report the wall-clock time and the memory consumed by {\sc Snipal} and Gurobi, respectively.}

\begin{table}[h]
\begin{footnotesize}
\begin{center}
\caption{Numerical results for random sparse LPs with ${\tt Tol} = 10^{-8}$. }
\label{table-random}

{
\begin{tabular}{|c|c|c|c|r|c|r|}
 \hline
\multicolumn{1}{|c}{$m$} &\multicolumn{1}{|c}{$n$} &\multicolumn{1}{|c}{$d$}
&\multicolumn{1}{|c}{$\begin{array}{c}\mbox{{\sc Snipal}} \\ \mbox{iter (itssn)}\end{array}$}  
 &\multicolumn{1}{|c|}{$\begin{array}{c}\mbox{{\sc Snipal}} \\ \mbox{time (s) (RAM)}\end{array}$}  
  &\multicolumn{1}{|c|}{$\begin{array}{c}\mbox{Gurobi} \\ \mbox{barrier iter}\end{array}$}  
 &\multicolumn{1}{|c|}{$\begin{array}{c}\mbox{Gurobi} \\ \mbox{time (s) (RAM)}\end{array}$}  
\\
\hline
2e3 & 1e5 & 0.025 & 4 (18)  &3.5 (0.8GB) & 7 &5.2 (1.0GB)
\\
  &  & 0.050 & 4 (18)  &6.8 (0.8GB)  & 7 &10.5 (1.3GB)
\\ \hline
5e3 & 1e5 & 0.025 & 4 (15)  &15.0 (1.5GB)  & 7 &22.4 (1.7GB)
\\
 &  & 0.050 & 4 (15)  &31.1 (1.8GB)   & 7 &46.1 (2.6GB)
\\ \hline
10e3 & 1e5 & 0.025 & 5 (24)  &  50.6 (3.2GB)  & 8 &96.6 (3.2GB)
\\
 &  & 0.050 & 5 (24)  &101.5  (5.3GB)  & 8  & 181.6 (6.0GB)
\\ \hline
1e3 & 1e6 & 0.025 & 5 (30)  &  10.3 (2.0GB) & 7 &22.2 (4.1GB)
\\
 &  & 0.050 & 5 (29)  &18.9 (3.2GB)   & 7 & 40.0 (6.0GB)
\\ \hline
2e3 & 1e6 & 0.025 & 6 (32)  &  27.2 (3.2GB) & 7 & 52.2 (6.1GB)
\\
 &  & 0.050 & 5 (28)  &53.1  (5.2GB)  &6 &  92.7 (9.6GB)
\\ \hline
5e3 & 1e6 & 0.025 & 5 (26)  &  91.8 (4.5GB) &7   & 184.1 (10.0GB)
\\ \hline
10e3 & 1e6 & 0.010 & 7 (40)  & 84.6 (4.3GB)  &7  & 194.4 (8.5GB)
\\ \hline
\end{tabular}
}
\end{center}
\end{footnotesize}
\end{table}

\subsection{Transportation problem}

In this problem, $s$ suppliers of $a_1,\ldots, a_s$ units of a certain goods must be transported to meet 
the demands $b_1,\ldots,b_t$ of $t$ customers. 
 Let the cost of transporting one unit of goods from supplier $i$  to customer $j$ be $c_{ij}$.
Then, the objective is to find a transportation plan denoted by $x_{ij}$ to solve the following LP:
\begin{eqnarray*}
\begin{array}{rl}
  \min & \sum_{i=1}^s \sum_{j=1}^t c_{ij} x_{ij} \\[5pt]
\mbox{s.t.} &  \sum_{j=1}^t x_{ij} = a_i, \quad i\in [s] \\[5pt]
&\sum_{i=1}^s x_{ij} = b_j,\quad j\in[t] \\[5pt]
& x_{ij} \geq 0,\; \forall\; i\in[s], \; j\in[t].
\end{array}
\end{eqnarray*}
In the above problem, we assume that $\sum_{i=1}^s a_i = \sum_{j=1}^t b_j$. 
{Note that this assumption is needed for the LP to be feasible.}
We can write the transportation LP compactly as follows:
\begin{eqnarray}
  \min \Big\{ \inprod{C}{X} \mid \cA(X) = [a; b], \; X \geq 0\Big\},
\end{eqnarray}
where
$$
\cA(X) = \left[ \begin{array}{c}
\hat{e}^T\otimes I_s
\\[5pt]
I_t\otimes e^T
\end{array}\right] {\rm vec}(X),
$$
$e\in \R^s$ and $\hat{e}\in\R^t$ are vector of all ones,
and ${\rm vec}(X)$ is the $st$-dimensional column vector obtained from $X$ by concatenating its columns
sequentially.

In {Table} \ref{tableTP}, we report the results for some randomly generated transportation instances.
For each pair of given $s,t$, we  generate a random transportation instance as follows:
\begin{verbatim}
 rng(`default'); M=abs(rand(s,t)); a=sum(M,2); b=sum(M,1)'; C=ceil(100*rand(s,t));
\end{verbatim}
{
Note that {
we turn off the presolve phase in Gurobi
as this phase is too time consuming (about 20--30\% of the total time)
and there is no benefit
in cutting down the computation time per iteration.}

We can observe that for this class of problems, {\sc Snipal} is able to
outperform the highly powerful {barrier method in} Gurobi by a factor of about {1--3 times}
in terms of computation times. Moreover,
our solver {\sc Snipal} consumed
less peak memory than Gurobi.
For the largest instance where the primal LP has 12,000 linear constraints and
27 millions variables, our solver is at least five times faster than {the barrier method in} Guorbi,
and it only needs 5.4GB of RAM whereas Gurobi required 12.8GB.
}

\begin{table}[h]
\begin{footnotesize}
\begin{center}
\caption{Numerical results for transportation LPs with ${\tt Tol} =10^{-8}$.}
\label{tableTP}
{
\begin{tabular}{|c|c|c|r|c|r|}
 \hline
 \multicolumn{1}{|c}{$s$} &\multicolumn{1}{|c}{$t$} 
&\multicolumn{1}{|c}{$\begin{array}{c}\mbox{{\sc Snipal}} \\ \mbox{iter (itssn)}\end{array}$}  
 &\multicolumn{1}{|c|}{$\begin{array}{c}\mbox{{\sc Snipal}} \\ \mbox{time (s) (RAM)}\end{array}$}  
  &\multicolumn{1}{|c|}{$\begin{array}{c}\mbox{Gurobi} \\ \mbox{barrier iter}\end{array}$}  
 &\multicolumn{1}{|c|}{$\begin{array}{c}\mbox{Gurobi} \\ \mbox{time (s) (RAM)}\end{array}$}  
\\
\hline
2000 & 3000 & 5 (17) & 18.3 (1.8GB) & 8 & 20.7 (4.8GB) 
\\
2000 & 4000 & 5 (18) & 22.0 (2.1GB)  &8 & 32.6 (6.5GB) 
\\
2000 & 6000 & 5 (18) & 34.2 (3.4GB) & 8  &59.5 (8.9GB) 
\\
\hline
3000 & 4500 & 5 (17)  & 40.4  (3.5GB) &8 &61.6 (9.2GB)  
\\
3000 & 6000 & 5 (18)  & 53.4 (4.0GB) &8 &93.9 (10.3GB)  
\\
3000 & 9000 &  5 (20)   & 65.1  (5.4GB)    &7 &191.1   (12.8GB)   
\\
\hline
\end{tabular}
}
\end{center}
\end{footnotesize}
\end{table}

\subsection{Generalized transportation problem}

The Generalized Transportation Problem (GTP) was introduced by Fergusan and Dantzig
\cite{GTP-Dantzig}
in their study of an aircraft routing problem. Eisemann and Lourie \cite{GTP-Eisemann}
applied it to the
machine loading problem.
 In that problem,
there are $m$ types of machines which can produce $n$ types of products
such that machine $i$ would take $h_{ij}$ hours at the cost of
$c_{ij}$ to produce one unit of product $j$.
It is
assumed that machine $i$ is available for at most $a_i$ hours, and the
demand for product $j$ is $b_j$. The problem is to determine
$x_{ij}$, the amount of product $j$ to be
produced on machine $i$ during the planning period
so that the total cost is minimized, namely,
\begin{eqnarray*}
\begin{array}{rl}
  \min & \sum_{i=1}^s \sum_{j=1}^t c_{ij} x_{ij} \\[5pt]
\mbox{s.t.} &  \sum_{j=1}^t h_{ij} x_{ij} = a_i, \quad i\in [s] \\[5pt]
&\sum_{i=1}^s x_{ij} = b_j,\quad j\in [t] \\[5pt]
& x_{ij} \geq 0,\; \forall\; i\in [s], \; j \in [t].
\end{array}
\end{eqnarray*}
In addition to assuming, {similar to the transportation problem in the previous subsection,} that $\sum_{i=1}^s a_i = \sum_{j=1}^t b_j$,
{we also apply the normalization} $\sum_{i=1}^s \sum_{j=1}^t h_{ij} = st.$

{Table \ref{tableGTP} presents the results for randomly generated
generalized transportation LPs where $a,b,c$ are generated as in the
last subsection. The weight matrix $H = (h_{ij})$ is generated by
setting {\tt  H = rand(s,t); H = (s*t/sum(sum(H)))*H.}
We can observe that  {\sc Snipal} can be up to 3 times faster than
{the barrier method in}
Gurobi when the problems are large.
}

\begin{table}[h]
\begin{footnotesize}
\begin{center}
\caption{Numerical results for generalized transportation LPs with ${\tt Tol}=10^{-8}$.}
\label{tableGTP}
{
\begin{tabular}{|c|c|c|r|c|r|}
 \hline
  \multicolumn{1}{|c}{$s$} &\multicolumn{1}{|c}{$t$} 
&\multicolumn{1}{|c}{$\begin{array}{c}\mbox{{\sc Snipal}} \\ \mbox{iter (itssn)}\end{array}$}  
 &\multicolumn{1}{|c|}{$\begin{array}{c}\mbox{{\sc Snipal}} \\ \mbox{time (s) (RAM)}\end{array}$}  
  &\multicolumn{1}{|c|}{$\begin{array}{c}\mbox{Gurobi} \\ \mbox{barrier iter}\end{array}$}  
 &\multicolumn{1}{|c|}{$\begin{array}{c}\mbox{Gurobi} \\ \mbox{time (s) (RAM)}\end{array}$}  
\\
\hline
2000 &3000 & 5 (19)  & 22.4 (1.6GB) &7 &19.4 (2.9GB) 
\\
2000 & 4000 & 5 (18)  & 27.1  (2.5GB) &8  &32.7 (5.5GB) 
\\
2000 & 6000 & 5 (18)  & 40.6 (3.6GB)  &8  &61.0 (8.0GB) 
\\
\hline
3000 & 4500 & 5 (18)  & 48.4 (3.4GB)  &8 &59.7 (6.0GB) 
\\
3000 & 6000 & 5 (18)  & 63.5 (4.0GB)  &8  &90.0 (12.9GB) 
\\
3000 & 9000 & 5 (19)  & 85.3  (5.8GB)  &7 & 258.0 (13.1GB) 
\\
\hline
\end{tabular}
}
\end{center}
\end{footnotesize}
\end{table}

\subsection{Covering and packing LPs}

Given a nonnegative matrix $A\in \R^{m\times n}$ and cost vector
$c\in \R^n_+$, the covering and packing LPs are defined by
\begin{eqnarray*}
 \mbox{(Covering)} && \min\Big\{ \inprod{c}{x} \mid A x \geq e, x\geq 0 \Big\}
 \\[5pt]
  \mbox{(Packing)} && \min\Big\{ \inprod{-c}{x} \mid A x \leq e, x\geq 0 \Big\}.
\end{eqnarray*}
It is easy to see that by adding a slack variable, the above problems can be
converted into the standard form expressed in (P).

In our numerical experiments in Table \ref{tableCP}, we generate  $A$ and $c$ randomly as follows:
\begin{center}
\begin{tt}
  rng(`default');  c = rand(n,1);  A = sprand(m,n,den);  A = round(A);
\end{tt}
\end{center}
{
Table \ref{tableCP} presents the numerical
performance of  {\sc Snipal} versus Gurobi on
some randomly generated large scale covering and packing LPs.
As we can observe, {\sc Snipal} is competitive against {the barrier method in} Gurobi
for solving these large scale LPs, and the former can
be up to {2.9 times faster than the barrier method in Gurobi.}
}

\begin{table}[h]
\begin{footnotesize}
\begin{center}
\caption{Numerical results for covering and packing LPs with ${\tt Tol} =10^{-8}$.}
\label{tableCP}
{
\begin{tabular}{|c|c|c|c|c|r|c|r|}
 \hline
 \multicolumn{1}{|c}{Type}
&\multicolumn{1}{|c}{$m$} &\multicolumn{1}{|c}{$n$}  &\multicolumn{1}{|c}{den}
&\multicolumn{1}{|c}{$\begin{array}{c}\mbox{{\sc Snipal}} \\ \mbox{iter (itssn)}\end{array}$}  
 &\multicolumn{1}{|c|}{$\begin{array}{c}\mbox{{\sc Snipal}} \\ \mbox{time (s)}\end{array}$}  
  &\multicolumn{1}{|c|}{$\begin{array}{c}\mbox{Gurobi} \\ \mbox{barrier iter}\end{array}$}  
 &\multicolumn{1}{|c|}{$\begin{array}{c}\mbox{Gurobi} \\ \mbox{time (s)}\end{array}$}  
\\
\hline
C & 1e3 & 5e5 & 0.2&  22 (148)  &49.8  &14 &62.1  
\\ \hline
C & 2e3 & 5e5 & 0.1 & 25 (151)  &103.3 &16  &90.6  
\\ \hline
C & 2e3 & 1e6 & 0.05 & 24 (160)  &90.4    &17 & 102.0 
\\ \hline
C & 3e3 & 5e6 & 0.02 & 24 (148)  &190.5    &22 &560.5  
\\
\hline \hline
P & 1e3 & 5e5 & 0.2 & 28 (160)  &  49.3  &12  &53.3  
\\ \hline
P & 2e3 & 5e5 & 0.1 & 29 (160)  & 97.0   &12 & 68.2  
\\ \hline
P & 2e3 & 1e6 & 0.05 & 30 (173)  &75.1    &15 &91.4  
\\
P & 3e3 & 5e6 & 0.02 &  26 (228)  & 259.8    & 20 & 500.2 
\\
\hline
\end{tabular}
}
\end{center}
\end{footnotesize}
\end{table}

\subsection{LPs arising from correlation clustering}

A correlation clustering problem \cite{BBC} is defined over an undirected graph
$G=(V,E)$ with $p$ nodes and edge weights $c_e \in \R$ (for each $e\in E$)
that is interpreted as a confidence measure of the similarity or dissimilarity
of the edge's end nodes.
In general, for $e=(u,v)\in E$, $c_e$ is given a negative value
if $u,v$ are dissimilar, and a positive value if $u,v$ are similar.
For the goal of finding a clustering that minimizes the disagreements,
the problem can be formulated as an integer programming
problem as follows. Suppose that
we are given a clustering  $\mathbb{S} =\{S_1,\ldots,S_N\}$ where each  $S_t\subset V$,
$t=1,\ldots,N$, denotes a cluster.
For each edge $e=(u,v)\in E$, set
$y_{e} = 0$ if $u,v\in S_t$ for some $t$, and  set $y_{e} =1$ otherwise.
Observe that $1-y_e$ is $1$ if $u,v$ are in the same cluster, and
$0$ if $u,v$ are in different clusters. Now define the constants
\begin{eqnarray*}
  m_e = |\min\{0,c_e\}|, \quad
       p_e = \max\{0,c_e\}.
\end{eqnarray*}
Then the cost of disagreements for the clustering $\mathbb{S}$ is given
by $\sum_{e\in E} m_e(1-y_e) + \sum_{e\in E} p_e y_e. $

A version of the correlation clustering
problem is to find a valid assignment (i.e., it satisfies the triangle inequalities)
of $y_e$ for all $e\in E$ to minimize the disagreements' cost.
We consider the relaxation of this integer program to get the following LP:
\begin{eqnarray*}
  \begin{array}{rl}
   \min & \sum_{(i,j)\in E} m_{ij}(1-y_{ij}) + \sum_{(i,j)\in E} p_{ij} y_{ij}
   \\[5pt]
   \mbox{s.t.} &
   - y_{ij} \leq 0, \;\; y_{ij} \leq 1\quad \forall\; (i,j)\in E
   \\[5pt]
   & -y_{ij} - y_{jk} + y_{ik} \leq 0 \quad \forall \; 1 \leq i < j < k\leq n,
   \;\mbox{such that}\;
   (i,j),(j,k),(i,k)\in E.
  \end{array}
\end{eqnarray*}
In the above formulation, we assumed that the edge set $E$ is a subset of
$\{ (i,j) \mid 1\leq i< j\leq p\}.$
Let $M$ be the  number of all possible triangles in $E$.
Define $\cT: \R^{|E|} \to \R^M$ to be the linear map that maps $y$ to all the
$M$ terms $-y_{ij} - y_{jk} + y_{ik}$ in the
triangle inequalities.
We can express the above LP in the dual form as follows:
\begin{eqnarray*}
  \inprod{m}{{\bf 1}} -\max \Big\{ \inprod{m-p}{y} \mid
  \left[\begin{array}{r}
   -I \\ I \\ \cT\end{array} \right] y \;\leq\; \left[ \begin{array}{c} 0
   \\ {\bf 1} \\ 0\end{array}\right]
   \Big\},
\end{eqnarray*}
and the corresponding primal LP is given
by
\begin{eqnarray}
  \inprod{m}{{\bf 1}} -
  \min \Big\{ \inprod{[0; {\bf 1};0]}{x} \mid [-I,\, I,\, \cT^*] x = m-p, \;
  x\in\R_+^{2|E|+M}\Big\}.
\label{eq-correlation-LP}
\end{eqnarray}
Observe that the primal LP has $|E|$ equality constraints and a large number of $2|E| + M$
variables.

In Table \ref{tableCC}, we evaluate the performance of our algorithm on
correlation clustering LPs on data that were used in \cite{VWG}.
{One can observe that for the LP problem \eqref{eq-correlation-LP}, our solver {\sc Snipal} is
much more efficient than the barrier method in Gurobi, and the former can be up to $117$ times faster for the
largest problem.
The main reason why {\sc Snipal} is able to outperform the barrier method in Gurobi lies in the 
fact that the former is able to make use of an iterative solver to solve the moderately well conditioned linear system
in \eqref{eq-H} rather efficiently in each semismooth Newton iteration, whereas for the latter, it has to rely on 
sparse Cholesky factorization to solve the associated normal equation and for this class of problems, 
computing the sparse Cholesky factorization is expensive. 
Under the column ``itminres'' in Table \ref{tableCC}, we report the average number of MINRES 
iterations needed to solve a single linear system of the form in \eqref{eq-H}. 
As one can observe, the average number of MINRES iterations is small compared to
the dimension of the linear system for all the tested instances. 
}

\begin{table}[h]
\begin{footnotesize}
\begin{center}
\caption{Numerical results for correlation clustering LPs with ${\tt Tol}=10^{-8}$.}
\label{tableCC}
{
\begin{tabular}{|c|c|c|c|c|r|c|r|}
 \hline
 \multicolumn{1}{|c}{Data}
&\multicolumn{1}{|c}{$p$} &\multicolumn{1}{|c}{$|E|$}  &\multicolumn{1}{|c}{$2|E|+M$}
&\multicolumn{1}{|c}{$\begin{array}{c}\mbox{{\sc Snipal}} \\ \mbox{iter (itssn $|$ itminres)}\end{array}$}  
 &\multicolumn{1}{|c|}{$\begin{array}{c}\mbox{{\sc Snipal}} \\ \mbox{time (s)}\end{array}$}  
  &\multicolumn{1}{|c|}{$\begin{array}{c}\mbox{Gurobi} \\ \mbox{barrier iter}\end{array}$}  
 &\multicolumn{1}{|c|}{$\begin{array}{c}\mbox{Gurobi} \\ \mbox{time (s)}\end{array}$}  
\\
\hline
 planted(5) & 200 & 19900 & 1353200 & 5 (70 $|$ 110.0) &38.1 & 37 &690.9 
 \\ \hline
planted(10) & 200 & 19900 & 1353200 & 6 (91 $|$ 146.5)  &  36.8  &49 &1146.6 
 \\ \hline
  planted(5) & 300 &44850 & 4544800 & 5 (86 $|$ 109.2)  & 170.2  &37 &8350.7
 \\ \hline
 planted(10) & 300 &44850 & 4544800 & 7 (127 $|$ 186.3)  & 158.0  &82  &18615.8 
 \\ \hline
stocks & 200 & 19900 & 1353200 &  5 (57 $|$ 147.7) & 57.8 &53 &1009.2 
\\ \hline
stocks & 300 & 44850 & 4544800 &  5 (75 $|$ 191.1) & 276.9 &60 &13797.0 
\\ \hline
\end{tabular}
}
\end{center}
\end{footnotesize}
\end{table}

\subsection{LPs  from  MIPLIB2010}

{
In this subsection, we evaluate the potential of {\sc Snipal} as a tool for solving 
general LPs with the characteristic that the number of linear constraints
are much smaller than the dimension of the variables. For this purpose, 
we consider the root-node LP relaxations of some mixed-integer programming
problems in the library MIPLIB2010 \cite{MIPLIB2010}.

Table \ref{tableMIPLIB} reports the performance of {\sc Snipal} against the 
barrier method in Gurobi for solving the LPs from the two sources mentioned
in the last paragraph
 to the accuracy level to $10^{-6}$. 
Note that we first use Gurobi's presolve
function to pre-processed the LPs. Then the pre-processed instances are used for
comparison with Gurobi's presolve capability turned off.
As one can observe, the barrier method in Gurobi performed much better
than {\sc Snipal}, with the former typically requiring less than 50 iterations to solve the LPs
while the latter typically needs hundreds of semismooth Newton iterations
except for a few problems such as {\tt datt256, neos-xxxx}, etc. 
Overall, the barrier method in Gurobi can be 10-50 times faster than {\sc Snipal}
on many of the tested instances, with the exception of {\tt ns2137859}.

The large number of semismooth Newton iterations needed by {\sc Snipal} 
to solve the LPs can be attributed to the fact that for most of the LP instances
tested here, the local superlinear convergent property of the semismooth Newton method
in solving the subproblems of the SPALM generally does not kick-in  
before a large number of initial iterations has been taken. 
From this limited set of tested LPs, we may conclude that
substantial numerical work must be done to improve the practical performance 
of {\sc Snipal} before it is competitive enough to solve 
general large scale sparse LPs.

\begin{footnotesize}
\begin{center}
\begin{longtable}{|l r|r| r|r|r|r|} 
\caption{Numerical results for some LPs  from MIPLIB2010 with ${\tt Tol}=10^{-6}$.}
\label{tableMIPLIB}
\\ \hline 
\mc{1}{|c}{problem}
&\mc{1}{c|}{$m$}
&\mc{1}{c|}{$n$}
&\mc{1}{c|}{it (itssn)} 
&\mc{1}{c}{time}
&\multicolumn{1}{|c|}{$\begin{array}{c}\mbox{Gurobi}\\ \mbox{barrier iter}\end{array}$}
&\multicolumn{1}{|c|}{$\begin{array}{c}\mbox{Gurobi}\\ \mbox{time}\end{array}$}
\\ 
\endfirsthead
\hline 
\mc{1}{|c}{problem}
&\mc{1}{c|}{$m$}
&\mc{1}{c|}{$n$}
&\mc{1}{c|}{it (itssn)} 
&\mc{1}{c}{time}
&\multicolumn{1}{|c|}{$\begin{array}{c}\mbox{Gurobi}\\ \mbox{barrier iter}\end{array}$}
&\multicolumn{1}{|c|}{$\begin{array}{c}\mbox{Gurobi}\\ \mbox{time}\end{array}$}
\\ \hline
\endhead

 \hline
\endfoot

 & &\mc{1}{c|}{} &\mc{1}{c|}{} &\mc{1}{c|}{} &\mc{1}{c|}{} &\mc{1}{c|}{}\\[-12pt] \hline 
\mc{1}{|r}{app1-2} 
&26850  &  107132 	 &33 (878) &54.33 & 16 &  0.73\\[2pt] 
\mc{1}{|r}{bab3} 
&22449  &  411334 	 &40 (789) &139.43 & 37 &  6.41\\[2pt] 
\mc{1}{|r}{bley-xl1} 
& 746  &  7361 	 & 7 (114) & 1.06 & 21 &  0.21\\[2pt] 
\mc{1}{|r}{circ10-3} 
&2700  &  46130 	 & 7 (17) & 4.98 & 10 &  0.80\\[2pt] 
\mc{1}{|r}{co-100} 
&1293  &  22823 	 &39 (634) & 4.09 & 23 &  0.64\\[2pt] 
\mc{1}{|r}{core2536-691} 
&1895  &  12991 	 &11 (325) &15.35 & 25 &  0.83\\[2pt] 
\mc{1}{|r}{core4872-1529} 
&3982  &  18965 	 &16 (285) &34.37 & 22 &  1.52\\[2pt] 
\mc{1}{|r}{datt256} 
&9809  &  193639 	 & 3 (37) &31.24 &  5 &  1.69\\[2pt] 
\mc{1}{|r}{dc1l} 
&1071  &  34931 	 &16 (209) & 8.86 & 39 &  1.06\\[2pt] 
\mc{1}{|r}{ds-big} 
&1039  &  173026 	 &27 (623) &74.91 & 25 &  5.37\\[2pt] 
\mc{1}{|r}{eilA101.2} 
& 100  &  65832 	 &14 (88) & 7.71 & 21 &  0.96\\[2pt] 
\mc{1}{|r}{ivu06-big} 
&1177  &  2197774 	 &19 (236) &87.28 & 27 & 34.97\\[2pt] 
\mc{1}{|r}{ivu52} 
&2116  &  135634 	 &31 (559) &47.19 & 23 &  3.22\\[2pt] 
\mc{1}{|r}{lectsched-1-obj} 
&9246  &  34592 	 &28 (331) & 2.46 & 12 &  0.32\\[2pt] 
\mc{1}{|r}{lectsched-1} 
&6731  &  27042 	 & 7 (15) & 0.26 &  5 &  0.15\\[2pt] 
\mc{1}{|r}{lectsched-4-obj} 
&2592  &  9716 	 &22 (94) & 0.76 &  7 &  0.06\\[2pt] 
\mc{1}{|r}{leo2} 
& 539  &  11456 	 &24 (106) & 0.68 & 22 &  0.16\\[2pt] 
\mc{1}{|r}{mspp16} 
&4065  &  532749 	 &26 (54) &68.67 & 14 & 59.88\\[2pt] 
\mc{1}{|r}{n3div36} 
&4450  &  25052 	 &20 (75) & 0.72 & 27 &  0.25\\[2pt] 
\mc{1}{|r}{n3seq24} 
&5950  &  125746 	 &14 (71) &20.25 & 15 &  6.42\\[2pt] 
\mc{1}{|r}{n15-3} 
&29234  &  153400 	 &22 (475) &63.86 & 30 &  4.30\\[2pt] 
\mc{1}{|r}{neos13} 
&1826  &  22930 	 &30 (154) & 9.83 & 22 &  0.23\\[2pt] 
\mc{1}{|r}{neos-476283} 
&9227  &  20643 	 &22 (495) &174.60 & 14 & 10.58\\[2pt] 
\mc{1}{|r}{neos-506428} 
&40806  &  200653 	 & 4 (8) & 7.86 & 16 &  0.96\\[2pt] 
\mc{1}{|r}{neos-631710} 
&3072  &  169825 	 & 4 (9) &13.90 &  5 &  0.54\\[2pt] 
\mc{1}{|r}{neos-885524} 
&  60  &  21317 	 & 4 (8) & 0.84 & 12 &  0.11\\[2pt] 
\mc{1}{|r}{neos-932816} 
&2568  &  8932 	 & 7 (17) & 2.88 & 10 &  0.14\\[2pt] 
\mc{1}{|r}{neos-941313} 
&12919  &  129180 	 & 6 (17) & 5.71 & 10 &  0.42\\[2pt] 
\mc{1}{|r}{neos-1429212} 
&8773  &  42620 	 &37 (541) &118.34 & 28 &  6.91\\[2pt] 
\mc{1}{|r}{netdiversion} 
&99482  &  208447 	 &33 (324) &173.07 & 15 &  5.38\\[2pt] 
\mc{1}{|r}{ns1111636} 
&12992  &  85327 	 & 4 (38) & 7.85 & 16 &  0.79\\[2pt] 
\mc{1}{|r}{ns1116954} 
&11928  &  141529 	 & 2 (6) &125.54 & 11 & 23.86\\[2pt] 
\mc{1}{|r}{ns1688926} 
&16489  &  41170 	 &26 (160) &150.80 & 88 & 12.68\\[2pt] 
\mc{1}{|r}{ns1904248} 
&38184  &  222489 	 & 3 (6) & 6.00 &  6 &  0.96\\[2pt] 
\mc{1}{|r}{ns2118727} 
&7017  &  15853 	 &30 (1079) &20.63 & 24 &  0.41\\[2pt] 
\mc{1}{|r}{ns2124243} 
&19663  &  53716 	 &22 (122) &13.08 & 14 &  0.36\\[2pt] 
\mc{1}{|r}{ns2137859} 
&16357  &  49795 	 &11 (22) & 6.20 & 50 & 61.53\\[2pt] 
\mc{1}{|r}{opm2-z12-s7} 
&10328  &  145436 	 &13 (43) &19.13 & 17 & 16.19\\[2pt] 
\mc{1}{|r}{opm2-z12-s14} 
&10323  &  145261 	 &12 (36) &19.46 & 16 & 15.39\\[2pt] 
\mc{1}{|r}{pb-simp-nonunif} 
&11706  &  146052 	 & 2 (4) & 2.08 & 10 &  0.68\\[2pt] 
\mc{1}{|r}{rail507} 
& 449  &  23161 	 &14 (240) & 1.95 & 23 &  0.25\\[2pt] 
\mc{1}{|r}{rocII-7-11} 
&5534  &  25590 	 &20 (73) & 2.26 & 17 &  0.35\\[2pt] 
\mc{1}{|r}{rocII-9-11} 
&8176  &  37159 	 &22 (106) & 3.85 & 17 &  0.51\\[2pt] 
\mc{1}{|r}{rvb-sub} 
& 217  &  33200 	 &24 (157) & 1.67 & 11 &  0.56\\[2pt] 
\mc{1}{|r}{shipsched} 
&5165  &  22806 	 &16 (35) & 0.73 & 10 &  0.24\\[2pt] 
\mc{1}{|r}{sp97ar} 
&1627  &  15686 	 &26 (264) & 2.63 & 26 &  0.33\\[2pt] 
\mc{1}{|r}{sp98ic} 
& 806  &  11697 	 &27 (155) & 1.63 & 25 &  0.25\\[2pt] 
\mc{1}{|r}{stp3d} 
&95279  &  205516 	 &71 (2892) &228.56 & 22 & 11.75\\[2pt] 
\mc{1}{|r}{sts729} 
& 729  &  89910 	 & 2 (4) & 0.72 &  3 &  0.26\\[2pt] 
\mc{1}{|r}{t1717} 
& 551  &  16428 	 &22 (141) & 1.59 & 13 &  0.22\\[2pt] 
\mc{1}{|r}{tanglegram1} 
&32705  &  130562 	 & 2 (4) & 0.90 &  5 &  0.30\\[2pt] 
\mc{1}{|r}{van} 
&7360  &  36736 	 & 4 (8) & 7.39 & 15 &  3.63\\[2pt] 
\mc{1}{|r}{vpphard} 
&9621  &  22841 	 & 3 (6) & 2.90 &  9 &  0.90\\[2pt] 
\mc{1}{|r}{vpphard2} 
&13085  &  28311 	 & 4 (6) & 2.77 &  7 &  0.64\\[2pt] 
\mc{1}{|r}{wnq-n100-mw99-14} 
& 594  &  10594 	 &24 (119) & 0.73 & 15 &  0.22\\[2pt] 

\hline
\end{longtable}
\end{center}
\end{footnotesize}

}
\section{Conclusion}

In this paper, we proposed a method called {\sc Snipal}
{targeted at solving large scale LP problems where the dimension $n$ of the 
decision variables  is much larger than the number $m$ of equality constraints.}
 {\sc Snipal} is an inexact proximal augmented Lagrangian method where the inner subproblems are {solved} via an efficient
semismooth Newton method. By connecting the inexact proximal augmented Lagrangian method with the preconditioned proximal point algorithm, we are able to show the global and local {asymptotic superlinear convergence} of {\sc Snipal}.
Our analysis also {reveals} that {\sc Snipal} can enjoy
a certain finite termination property.
To achieve high performance, we further {study various efficient approaches   for solving the large} linear systems in the semismooth Newton method. Our findings indicate that the linear systems involved in {\sc Snipal}
have uniformly bounded condition numbers, {in contrast to those 
 involved in an interior point algorithm which has unbounded condition numbers}. Building upon all
the aforementioned desirable properties, our algorithm {\sc Snipal}
has demonstrated a clear computational advantage in solving
{some classes of} large-scale  LP problems in the numerical experiments
when tested against the {barrier method in the 
powerful commercial LP solver Gurobi. 
However, when tested on some large sparse LPs available in the 
public domain, our {\sc Snipal} is not yet competitive against the barrier method in Gurobi on most of the 
test instances. Thus 
much work remains to be done to improve the practical efficiency of {\sc Snipal} and we leave
it as a future research project. 
}

\section*{Appendix}

Here we show that the dual of \eqref{eq:suby} with $\tau =0$ is given by
\eqref{eq-ppa-sub}.
Consider the augmented Lagrangian function
\begin{eqnarray*}
\inf_{y} L_{\sigma_k}(y;x^k) &= &
\inf_y \max_{x}\left\{
l(y;x) - \frac{1}{2\sigma}\norm{x-x^k}^2
\right\} 
= \max_x\left\{ - \frac{1}{2\sigma}\norm{x-x^k}^2
+ \inf_y l(y;x)
\right\} \\[5pt]
&=&  \max_x\Big\{ -\delta_K(x)-\inprod{c}{x} - \frac{1}{2\sigma}\norm{x - x^k}^2 \mid Ax = b \Big\},
\end{eqnarray*}
where $l(y;x) =  -b^T y - \inprod{x}{c - A^* y} - \delta_K(x)$ for any $(y,x)\in\R^m\times \R^n$.
The interchange of $\inf_y$ and $\max_x$ follows from the growth properties in $x$ of the ``minimaximand'' in question
\cite[Theorem 37.3]{rockafellar1970convex}. See also the proof of \cite[Proposition 6]{rockafellar1976augmented}.

\bibliographystyle{siamplain}

\begin{thebibliography}{TL 99}



\bibitem{BBC}
{\sc N. Bansal, A. Blum, and S. Chawla}, {\em  Correlation clustering}, IEEE
Symposium on Foundations of Computer Science, 2002.

\bibitem{bauschke2011convex}
{\sc H.~H. Bauschke and P.~L. Combettes}, {\em Convex Analysis and Monotone Operator Theory in Hilbert Spaces}, Springer, 2011.

\bibitem{BGZ-04}
{
{\sc L. Bergamaschi, J. Gondzio, and G. Zilli}, {\em Preconditioning indefinite systems in
interior point methods for optimization}, Computational Optimization and Applications, 28 (2004), pp. 149--171.
}

\bibitem{bonnans1995family}
{\sc J.~F. Bonnans, J.~Ch. Gilbert, C. Lemar{\'e}chal, and C.~A. Sagastiz{\'a}bal}, {\em A family of variable metric proximal methods},
Mathematical Programming, 68 (1995), pp. 15--47.

\bibitem{Burke1999variable}
{\sc J. V. Burke and M. Qian}, {\em A variable metric proximal point algorithm for monotone operators},
SIAM Journal on Control and Optimization, 37 (1999), pp. 353--375.

\bibitem{Burke1999local}
{\sc J. V. Burke and M. Qian}, {\em On the local super-linear convergence of a matrix secant implementation
	of the variable metric proximal point algorithm for monotone operators}, in Reformulation - Nonsmooth, Piecewise Smooth, Semismooth and Smoothing Methods, M. Fukushima and L. Qi, eds., Kluwer Academic Publishers, Norwell, MA, 1999, pp. 317--334.

\bibitem{Burke2000superlinear}
{\sc J. V. Burke and M. Qian}, {\em On the superlinear convergence of the variable metric proximal
	point algorithm using Broyden and BFGS matrix secant updating}, Mathematical Programming, 88 (2000), pp. 157--181.

\bibitem{Chai-Toh-07}
{\sc J. S. Chai and K.-C. Toh}, {\em Preconditioning and iterative solution of symmetric indefinite linear systems arising from interior point methods for linear programming},
Computational Optimization and Applications, 36 (2007), pp. 221--247.

\bibitem{Burer}
{\sc J. Chen and S. Burer}, {\em
A first-order smoothing technique for
a class of large-scale linear programs}, SIAM Journal on  Optimization, 24 (2014), pp. 598--620.

\bibitem{chen2018equivalent}
{\sc L. Chen, X. D. Li, D. F. Sun and K.-C. Toh},
{\em On the equivalence of inexact proximal ALM and ADMM for a class of convex composite programming}, Mathematical Programming, (2020), DOI:10.1007/s10107-019-01423-x.




\bibitem{Chen1999proximal}
{\sc X. Chen and M. Fukushima}, {\em Proximal quasi-Newton methods for nondifferentiable convex
optimization}, Mathematical Programming, 85 (1999), pp. 313--334.

\bibitem{Clarke1983Optimization}
{\sc F. Clarke}, {\em Optimization and Nonsmooth Analysis}, John Wiley and Sons, New York, 1983.

\bibitem{CMTH-16}
{\sc Y. Cui, K. Morikuni, T. Tsuchiya, and K. Hayami}, {\em Implementation of interior-point
methods for  LP based on Krylov subspace iterative solvers with inner-iteration preconditioning}, 
Computational Optimization and Applications, 74 (2019), pp. 143--176. 

\bibitem{Durier1988locally}
{\sc R. Durier}, {\em On locally polyhedral convex functions}, in Trends in Mathematical Optimization
(Irsee, 1986), Internat. Schriftenreihe Numer. Math., 84, Birkh\"{a}user, Basel, 1988, pp. 55--66.


\bibitem{eckstein1993nonlinear}
{\sc J. Eckstein}, {\em Nonlinear proximal point algorithms using Bregman functions, with applications to convex programming}, Mathematics of Operations Research, 18 (1993), pp. 202--226.

\bibitem{Eckstein1992Douglas}
{\sc J. Eckstein and D. P. Bertsekas}, {\em On the Douglas-Rachford splitting method and the proximal
point algorithm for maximal monotone operators}, Mathematical Programming, 55 (1992), pp. 293--318.



\bibitem{EGM-05}
{\sc Yu.~G. Evtushenko, A.~I. Golikov, and N. Mollaverdy}, {\em
Augmented Lagrangian method for large-scale
linear programming problems}, Optimization Methods and Software, 20 (2005), pp. 515--524.

\bibitem{GTP-Dantzig}
{\sc A.~R. Fergusan and G.~B. Dantzig}, {\em The allocation of aircrafts to routes - an example of linear programming under uncertain demand},
Management Science, 3 (1956).

\bibitem{Fischer1996finite}
{\sc A. Fischer and C. Kanzow}, {\em On finite termination of an iterative method for linear complementarity
problems}, Mathematical Programming, 74 (1996), pp. 279--292.

\bibitem{GTP-Eisemann} {\sc K. Eisemann and J.R. Lourie},
{\em The machine loading problem}, IBM 704 Program,
BML-1, IBM Application Library, New York, 1959.

\bibitem{Gondzio-08}
{\sc G. Al-Jeiroudi, J. Gondzio and J.A.J. Hall}, {\em
Preconditioning indefinite systems in interior point methods for large scale linear optimization},
Optimization Methods and Software, 23 (2008), pp. 345--363.

\bibitem{Hiriart1984geeralized}
{\sc J.-B. Hiriart-Urruty, J.-J. Strodiot, and V.~H. Nguyen}, {\em Generalized Hessian matrix and second-order optimality
conditions for problems with $C^{1,1}$ data}, Applied Mathematics and Optimization, 11 (1984), pp.~43--56.

\bibitem{Kanzow2003minimum}
{\sc C. Kanzow, H. Qi, and L. Qi}, {\em On the minimum norm solution of linear programs}, Journal of Optimization Theory and Applications, 116 (2003), pp.~333--345.

\bibitem{Klatte1995nonsmooth}
{\sc D. Klatte and B. Kummer}, {\em Nonsmooth Equations in Optimization, Regularity, Calculus, Methods and Applications},
Kluwer Academic Publishers, Dordrecht, the Netherlands, 2002.


\bibitem{LST-QSDPNAL}
{\sc X.~D. Li, D.~F. Sun, and K.-C. Toh}, {\em QSDPNAL: A two-phase augmented Lagrangian method for convex quadratic semidefinite programming}, Mathematical Programming Computation, 10 (2018), pp.~703--743.

\bibitem{LST-17}
{\sc X.~D. Li, D.~F. Sun, and K.-C. Toh}, {\em On the efficient computation of a generalized Jacobian of the projector over the Birkhoff polytope}, Mathematical Programming, 179 (2020), pp.~419--446.

\bibitem{luque1984asymptotic}
{\sc F.~J. Luque}, {\em Asymptotic convergence analysis of the proximal point
	algorithm}, SIAM Journal on Control and Optimization, 22 (1984), pp.~277--293.

\bibitem{Mangasarian-88}
{\sc R. De Leone and O. L. Mangasarian}, {\em Serial and parallel solution of large scale linear
	programs by augmented Lagrangian successive overrelaxation}, in A. Kurzhanski,
K. Neumann, and D. Pallaschke, editors, Optimization, Parallel Processing and Applications,
pp. 103--124, Berlin, 1988. Springer-Verlag.


\bibitem{Mangasarian-79}
{\sc O. L. Mangasarian and R. R. Meyer}, {\em Nonlinear perturbation of linear
	programs}, SIAM J. Control and Optimization, 17 (1979), pp. 745--752.

\bibitem{Mangasarian-04} {\sc O.~L. Mangasarian},
{\em A Newton method for linear programming}, Journal of Optimization
Theory and Applications, 121 (2004), pp. 1--18.

\bibitem{MIPLIB2010} {\em MIPLIB -- C the Mixed Integer Programming LIBrary}, available at 
{\tt http://miplib2010.zib.de/}.


\bibitem{Oliveira-05} {\sc A. R. L. Oliveira and D. C. Sorensen},
     {\em A new class of preconditioners for large-scale linear
     systems from interior point methods for linear programming},
     Linear Algebra and its Applications, 394, 1-24, 2005.



\bibitem{polyak1987introduction}
{\sc B.~T. Polyak}, {\em Introduction to Optimization}, Optimization Software Inc., New York, 1987.

\bibitem{Parente2008class}
{\sc L. A. Parente, P. A. Lotito, and M. V. Solodov},
{\em A class of inexact variable metric proximal point algorithms},
SIAM Journal on Optimization, 19 (2008), pp. 240--260.

\bibitem{Qi1995preconditioning}
{\sc L. Qi and X. Chen}, {\em A preconditioning proximal Newton's method for nondifferentiable convex
	optimization}, Mathematical Programming, 76 (1995), pp. 411--430.

\bibitem{robinson1976implicit}
{\sc S. M. Robinson}, {\em An implicit-function theorem for generalized variational inequalities}, Technical summary report no. 1672, Mathematics Research Center, University of Wisconsin-Madison, (1976);
available from National Technical Information Service under Accession No. ADA031952.

\bibitem{robinson1981some}
{\sc S. M. Robinson}, {\em Some continuity properties of polyhedral multifunctions}, in Mathematical
Programming at Oberwolfach, Math. Program. Stud., Springer, Berlin, Heidelberg, 1981, pp. 206--214.

\bibitem{rockafellar1970convex}
{\sc R.~T. Rockafellar}, {\em Convex Analysis}, Princeton University Press, Princeton, N.J., 1970.

\bibitem{rockafellar1976monotone}
{\sc R.~T. Rockafellar}, {\em Monotone operators
	and the proximal point algorithm}, SIAM Journal on Control and Optimization,
14 (1976), pp.~877--898.

\bibitem{rockafellar1976augmented}
{\sc R.~T. Rockafellar}, {\em Augmented {L}agrangians and applications of the
	proximal point algorithm in convex programming}, Mathematics of operations
research, 1 (1976), pp.~97--116.

\bibitem{rockafellar2009variational}
{\sc R. T. Rockafellar and R. J.-B. Wets}, {\em Variational Analysis}, Springer, New York, 2009.


\bibitem{Saad} { {\sc Y. Saad}, {\rm Iterative Methods}, PWS Publishing Company, Boston, 1996. }

\bibitem{Schork-Gondzio} { {\sc L. Schork and J. Gondzio}, {\em
Implementation of an Interior Point Method
with Basis Preconditioning}, Mathematical Programming Computation, (2020), DOI:10.1007/s12532-020-00181-8.}



\bibitem{sun1998finite}
{\sc D. F. Sun, J. Y. Han, and Y. Zhao}, {\em On the finite termination of the damped-newton
algorithm for the linear complementarity problem}, Acta Mathematica Applicatae Sinica, 21 (1998), pp. 148--154.



\bibitem{VWG}
{\sc N. Veldt, A. Wirth, and D. Gleich}, {\em Correlation clustering with low-rank matrices},
Proceedings of the 26th International Conference on World Wide Web, 2017, pp. 1025--1034.


\bibitem{Wright-90} {\sc S.~J. Wirght}, {\em Implementing proximal point methods for
linear programming}, Journal of Optimization Theory and Applications, 65 (1990), pp. 531--554.

\bibitem{Yang2015SDPNAL+}
{\sc L.~Q. Yang, D.~F. Sun, and K.-C. Toh}, {\em SDPNAL+: a majorized semismooth Newton-CG augmented Lagrangian method for semidefinite programming with nonnegative constraints}, Mathematical Programming Computation, 7 (2015), pp. 331--366.

\bibitem{YZHRD}
E.-H. Yen, K. Zhong, C.-J. Hsieh, P. K. Ravikumar, and I. S. Dhillon, {\em  Sparse linear
programming via primal and dual augmented coordinate descent},  Advances in Neural Information
Processing Systems, 2015, pp. 2368--2376.

\bibitem{Zhao2010Newton}
{\sc X.-Y. Zhao, D. F. Sun, and K.-C. Toh}, {\em A Newton-CG augmented Lagrangian method for
semidefinite programming}, SIAM Journal on  Optimization, 20 (2010), pp. 1737--1765.
\end{thebibliography}

\end{document}